\documentclass[a4paper,11pt]{article}

\usepackage[T1]{fontenc}
\usepackage{lmodern}

\usepackage{dsfont}

\usepackage[latin1]{inputenc}
\usepackage{amsmath}
\usepackage{amsthm}
\usepackage{amssymb}
\usepackage{mathrsfs}
\usepackage{graphicx}
\usepackage[all]{xy}
\usepackage{hyperref}

\usepackage{makeidx}

\usepackage{stmaryrd}
\usepackage{caption}

\usepackage{abstract} 

\newtheorem{thm}{Theorem}[section]
\newtheorem{cor}[thm]{Corollary}
\newtheorem{claim}[thm]{Claim}
\newtheorem{fact}[thm]{Fact}

\newtheorem{lemma}[thm]{Lemma}
\newtheorem{prop}[thm]{Proposition}

\theoremstyle{definition}
\newtheorem{definition}[thm]{Definition}

\newtheorem{remark}[thm]{Remark}
\newtheorem{question}[thm]{Question}
\newtheorem{problem}[thm]{Problem}
\newtheorem{conj}[thm]{Conjecture}

\def\rquotient#1#2{%
	\makeatletter
	\raise.3ex\hbox{$#1$}/\lower.3ex\hbox{$#2$}%
	\makeatother
}	

\makeatletter
\newcommand{\subjclass}[2][2010]{%
	\let\@oldtitle\@title%
	\gdef\@title{\@oldtitle\footnotetext{#1 \emph{Mathematics subject classification.} #2}}%
}
\newcommand{\keywords}[1]{%
	\let\@@oldtitle\@title%
	\gdef\@title{\@@oldtitle\footnotetext{\emph{Key words and phrases.} #1.}}%
}
\makeatother

\newcommand{\Address}{{
		\bigskip
		\small
		
		\textsc{University of Montpellier\\ 
Institut Math\'ematiques Alexander Grothendieck\\
Place Eug\`ene Bataillon\\
34090 Montpellier (France)}\par\nopagebreak
		\textit{E-mail address}: \texttt{anthony.genevois@umontpellier.fr}
		
}}

\makeindex

\title{Examples of cubulable groups with fixed-point properties}
\date{\today}
\author{Anthony Genevois}


\begin{document}

\maketitle

\begin{abstract}
For every $n \geq 1$, let $(\mathrm{FW}_n)$ denote the fixed-point property for median graphs of cubical dimension $n$ (or equivalently, for CAT(0) cube complexes of dimension $n$). In this article, we construct explicit examples of groups satisfying $(\mathrm{FW}_n)$ but with good cubical properties in higher dimensions. First, we prove that, for a finitely generated group $G$ with no non-abelian free subgroup, $G$ satisfies $(\mathrm{FW}_n)$ if and only if no subgroup $H \leq G$ of index $\leq n$ can be mapped to $\mathbb{D}_\infty$ with an infinite image. For instance, the affine Coxeter group $\tilde{A}_n$ satisfies $(\mathrm{FW}_n)$ but not $(\mathrm{FW}_{n+1})$. In another direction, we investigate virtually graph products of finite groups. As an application of our constructions, we find explicit examples, for every $n \geq 1$, of acylindrically hyperbolic groups that are cocompactly cubulable but satisfy $(\mathrm{FW}_n)$. Several conjectures and open questions are included. 
\end{abstract}

\tableofcontents

\section{Introduction}

As a counterpart of the leitmotiv of geometric group theory that consists in studing groups acting in some interesting way on specific classes of metric spaces, the search for properties and examples of groups with only strongly restricted actions on the same families of spaces is a natural problem. In this article, we focus our attention to median graphs, also known as one-skeletons of CAT(0) cube complexes\footnote{Despite the fact that median graphs and one-skeletons of CAT(0) cube complexes define the same objects, we motivate in \cite{MedianVsCC} the idea that thinking about median graphs is more natural and more efficient. Therefore, we use this terminology in the article.}. Thinking about median graphs as a natural interpolation between (simplicial) trees and $\ell^1$-spaces, fixed-point properties on median graphs provide a natural interpolation between Serre's fixed-point property $(\mathrm{FA})$ on trees and a discrete version of Kazhdan's property $(\mathrm{T})$ known as $(\mathrm{FW})$ (as $(\mathrm{T})$ turns out to be equivalent, for locally compact second countable groups, to the fixed-point property on median spaces \cite{medianviewpoint}). 

$$\begin{array}{ccccccccccc}
(\mathrm{FA}) & &&&&&&&& & (\mathrm{T}) \\ \Updownarrow & &&&&&&&& & \Downarrow \\ (\mathrm{FW}_1) & \Leftarrow & (\mathrm{FW}_2) & \Leftarrow & \cdots & \Leftarrow & (\mathrm{FW}_\mathrm{fin}) & \Leftarrow & (\mathrm{FW}_\omega) & \Leftarrow & (\mathrm{FW})
\end{array}$$

\noindent
Here, $(\mathrm{FW})$ (resp. $(\mathrm{FW}_\omega)$, $(\mathrm{FW}_\mathrm{fin})$, $(\mathrm{FW}_n)$ for some $n \geq 1$) refers to the fixed-point property on median graphs (resp. with no infinite cube, of finite cubical dimension, of cubical dimension $n$). We emphasize that, given a specific family $\mathcal{X}$ of median graphs, a group $G$ satisfies the fixed-point property on $\mathcal{X}$ if every action of $G$ on every graph $X \in \mathcal{X}$ stabilises a cube (or equivalently, fixes a vertex in the cubical subdivision of $X$, or fixes a point in the cube-completion of $X$, or has bounded orbits).

\paragraph{State of the art.} Let us summarise what is known about $(\mathrm{FW})$ and its relatives.

 The fact that $(\mathrm{T})$ implies $(\mathrm{FW})$ is proved in \cite{MR1459140}. Consequently, groups such as $\mathrm{SL}(n,\mathbb{Z})$ for $n \geq 3$ and $\mathrm{Out}(\mathbb{F}_n)$ for $n \geq 4$ satisfy $(\mathrm{FW})$. The reverse implication is however not true. For instance, it is shown in \cite{CornulierCommensurated} that $\mathrm{SL}(2,\mathbb{Z}[\sqrt{2}])$ satisfies $(\mathrm{FW})$ but not $(\mathrm{T})$. It does not satisfies $(\mathrm{T})$ because, as a lattice in $\mathrm{SL}(2,\mathbb{R})^2$, it is a-T-menable. And it satisfies $(\mathrm{FW})$ because it is boundedly generated by two distorted abelian subgroups, namely the subgroups of unipotent lower and upper triangular matrices. Distortion is the key phenomenon here and illustrates a fundamental difference between median and Hilbertian geometry. It has been highlighted in \cite{arXiv:0705.3386} in order to show that some Baumslag-Solitar groups provide examples of a-T-menable groups that do not act properly on median graphs. 
	
 It is conjectured in \cite{CornulierCommensurated} that the property $(\mathrm{FW})$ holds for every irreducible lattice in a connected semi-simple Lie group with no compact simple factor whose Lie algebra has $\mathbb{R}$-rank $\geq 2$. The example $\mathrm{SL}(2,\mathbb{Z}[\sqrt{2}])$ previously mentioned is one particular case of the conjecture. See \cite{MR3299841} for other positive results in this direction. It is also worth mentioning that \cite{MR3509968} proves that every irreducible lattice in semi-simple Lie group of rank $\geq 2$ satisfies $(\mathrm{FW}_\mathrm{fin})$. 
	
 It is proved in \cite{GLU} that purely elliptic actions of finitely generated groups on median graphs with no infinite cube are automatically elliptic. As a consequence, every finitely generated torsion group satisfies $(\mathrm{FW}_\omega)$. However, such groups may not satisfy $(\mathrm{FW})$. For instance, given an arbitrary finitely generated infinite torsion group $T$, the wreath product $\mathbb{Z}/2\mathbb{Z} \wr T$ naturally acts on $\bigoplus_T \mathbb{Z}/2\mathbb{Z}$, which can be identified with an infinite cube. From this observation, it can be deduced that some Burnside groups do not satisfy $(\mathrm{FW})$ (see the appendix of \cite{MR3786300} for more details, and the core of the article for an alternative construction). The most famous example of finitely generated infinite torsion group, namely the Grigorchuk group $\mathrm{Gr}$, is also of interest. From what we have already said, we know that $\mathrm{Gr}$ satisfies $(\mathrm{FW}_\omega)$. However, it does not satisfy $(\mathrm{FW})$. This follows from the existence of multi-ended Schreier graphs \cite{MR3027509} and Sageev's construction \cite{MR1347406}. In fact, it can be proved that $\mathrm{Gr}$ acts properly on an infinite cube \cite{GriCC}. 
	
 Motivated by the idea that median graphs of finite cubical dimension are ``built from hyperbolic spaces'', we proved in \cite{MR3918481} that $(\mathrm{FW}_\mathrm{fin})$ is implied by the hereditary property $(\mathrm{NL})$, which claims that no finite-index subgroup in the group under consideration admits an action on a geodesic hyperbolic space with a loxodromic element ($\mathrm{NL}$ stands for \emph{No Loxodromic}). Removing actions with loxodromic elements only allows elliptic and horocyclic actions, i.e.\ the types of actions that every infinite finitely generated group admits. As an application, it is proved in \cite{MR3918481} that Thompson's group $V$ satisfies $(\mathrm{FW}_\mathrm{fin})$. On the other hand, we know from \cite{MR2136028} that $V$ admits a proper action on a locally finite median graphs of infinite cubical dimension. Thus, $(\mathrm{FW}_\mathrm{fin})$ does not imply $(\mathrm{FW}_\omega)$. Other examples of hereditary $(\mathrm{NL})$ groups can be found in \cite{BFG} and references therein, including several Thompson-like groups.
	
 A useful strategy to construct groups satisfying $(\mathrm{FW}_n)$ for a specific $n \geq 1$ is given by the following Helly theorem: given a contractible topological space $X$ of topological dimension $\leq n$ and a collection contractible subspaces $S_1, \ldots, S_r \subset X$, if the intersection $S_{i_1} \cap \cdots \cap S_{i_{n+1}}$ is non-empty for all $1 \leq i_1, \ldots, i_{n+1} \leq r$, then $S_1 \cap \cdots \cap S_r \neq \emptyset$. As a consequence, if $G$ is a group acting on $X$ and if $G$ is generated by subgroups $H_1, \ldots, H_r$ having non-empty contractible fixed-point sets, then it suffices to verify that the subgroup $\langle H_{i_1}, \ldots, H_{i_{n+1}} \rangle$ fixes a point of $X$ for all $1 \leq i_1, \ldots, i_{n+1} \leq r$ in order to conclude that $G$ has a global fixed point. This strategy is known for trees \cite{MR1954121} and has been popularised for CAT(0) spaces by \cite{MR2555905}. More precisely, given an $n \geq 1$, \cite{MR2555905} introduces the fixed-point property $(\mathrm{FA}_n)$ for complete $n$-dimensional CAT(0) cell complexes and deduces from Helly theorem various fixed-point properties for arithmetic and Chevalley groups \cite{MR2555905}. Since then, the properties $(\mathrm{FA}_n)$ have also been investigated for Coxeter groups \cite{MR2263060}, mapping class groups \cite{MR2999128}, automorphism groups of free groups \cite{MR3426225}, and Thompson-like groups \cite{MR4028980}. Since cube-completions of median graphs admit CAT(0) metrics, $(\mathrm{FA}_n)$ implies $(\mathrm{FW}_n)$ for every $n \geq 1$. It is worth mentioning that there exist groups satisfying $(\mathrm{FW}_n)$ but not $(\mathrm{FA}_n)$. For instance, as mentioned by Corollary~\ref{cor:AffineCoxeter} below, the Euclidean Coxeter group $\widetilde{A}_n$ satisfies $(\mathrm{FW}_{n-1})$, but it clearly does not satisfy $(\mathrm{FA}_{n-1})$ as it acts properly and cocompactly on $\mathbb{R}^{n-1}$. 
	
 In Gromov's density model, it is known that random groups for high density satisfy Kazhdan's property $(\mathrm{T})$ \cite{MR1995802}, and a fortiori $(\mathrm{FW})$, but act properly and cocompactly on median graphs for low density \cite{MR2806688} (see also \cite{MR3359032, MR4549688, RandomCC}). But the cubical dimensions of the median graphs thus obtained remain to be investigated. In \cite{MR2755002}, it is proved that random groups (for arbitrary density) satisfy Serre's property $(\mathrm{FA})$, or equivalently $(\mathrm{FW}_1)$. The result has been recently improved in \cite{RandomFW}, proving that random groups satisfy $(\mathrm{FW}_2)$, and finally generalised in \cite{MR5000781} for arbitrary dimensions. Thus, for every $n \geq 1$, there exists torsion-free hyperbolic groups that are cocompactly cubulated but satisfy $(\mathrm{FW}_n)$. However, no explicit example is exhibited.
	
	Another source of hyperbolic groups with interesting fixed-point properties is given by the Rips construction from \cite{MR2276608}, where it is proved that, for every countable group $Q$, there exists a graphical C'(1/6) small cancellation group $G$ such that
$$1 \to N \to G \to Q \to 1$$
where $N$ is finitely generated and satisfies Kazhdan's property (T). In particular, $G$ satisfies a property $(\mathrm{FW}_\ast)$ if and only if $Q$ satisfies the same $(\mathrm{FW}_\ast)$ property. For instance, taking $Q:= \mathrm{SL}(2, \mathbb{Z}[\sqrt{2}])$, one gets a hyperbolic group $G$ satisfying $(\mathrm{FW})$ but not (T). 
	
Property $(\mathrm{FA})$ has been extensively studied in the literature, and we do not intend to record what is known on the subject since its introduction in \cite{MR1954121}. We only mention that it is known for many classical families of groups, including, among many others, mapping class groups of most surfaces, many linear groups, some Coxeter groups, Cremona groups, some wreath products, many automorphism of free products.

\medskip \noindent
Globally, property $(\mathrm{FW})$ and its relatives remain poorly understood, and many natural questions on the subject have not been investigated. We record in Section~\ref{section:Questions} some conjectures and open problems.

\paragraph{Main results.} In this article, our goal is to construct explicit examples of groups with fixed-point properties in some dimensions but good actions in higher dimensions. First of all, we consider groups that do not contain non-abelian free subgroups. It is worth mentioning that virtually abelian groups are already known to provide examples of groups that do not act properly and cocompactly on median graphs but that contain finite-index subgroups admitting such actions \cite{MR4298722, MR3245140}. Here, we show that these groups, and more generally groups with no $\mathbb{F}_2$, also provide examples with interesting fixed-point properties. More precisely, we prove the following characterisation:

\begin{thm}\label{thmIntro:VirtuallyAbelian}
Let $G$ be a finitely generated group with no $\mathbb{F}_2$. Given an $n \geq 1$, $G$ satisfies $(\mathrm{FW}_n)$ if and only if there is no subgroup $H \leq G$ of index $\leq n$ with a morphism $H \to \mathbb{D}_\infty$ whose image is infinite.
\end{thm}

\noindent
See Theorem~\ref{thm:VirtuallyAbelian} for a more detailed statement. As concrete examples, it follows that: 
\begin{itemize}
	\item The affine Coxeter group $\widetilde{A}_n$ satisfies $(\mathrm{FW}_{n})$; see Corollary~\ref{cor:AffineCoxeter}. 
	\item Given a transitive permutation group $F \curvearrowright [n]:=\{1, \ldots, n\}$, the permutational wreath product $\mathbb{D}_\infty \wr_{[n]} F$ satisfies $(\mathrm{FW}_{n-1})$; see Corollary~\ref{cor:WreathFW}.
	\item Simple amenable groups, such as those constructed in \cite{MR3071509}, satisfy $(\mathrm{FW})$. 
	\item Some branch groups, such as the alternate mother group studied in \cite{MR3330165}, satisfy~$(\mathrm{FW})$.
\end{itemize}
In fact, $\widetilde{A}_n$ (resp.\ $\mathbb{D}_\infty \wr_{[n+1]} \mathrm{Sym}([n+1])$) does not only admit an action without global fixed point on a median graph of cubical dimension $n+1$, but it admits a proper (resp.\ proper and cocompact) action on such a median graph. Thus, there is brutal gap between the behaviours of the group in dimensions $n$ and $n+1$. 

\medskip \noindent
Next, we investigate examples in another direction, namely groups with some negative curvature instead of amenable-like groups. See Conjecture~\ref{Conj}, and the discussion related to Conjecture~\ref{ConjSC}, for some motivation. 

\medskip \noindent
Recall that, given a graph $\Gamma$ and a collection of groups $\mathcal{G}= \{ G_u, u \in V(\Gamma)\}$ indexed by the vertices of $\Gamma$, the \emph{graph product} $\Gamma \mathcal{G}$ is the group having
$$\langle G_u, \ u \in V(\Gamma) \mid [G_u,G_v]=1, \ \{u,v\} \in E(\Gamma) \rangle$$
as a relative presentation, where $V(\Gamma)$ and $E(\Gamma)$ denote the vertex- and edge-sets of $\Gamma$, and where $[G_u,G_v]=1$ is a shorthand for: $[g,h]=1$ for all $g \in G_u$ and $h \in G_v$. If all the vertex-groups $G_u$ are isomorphic to the same group $G$, we denote by $\Gamma G$ or $\Gamma(G)$ the graph product $\Gamma \mathcal{G}$. Theorems~\ref{thm:GraphProdFW} and~\ref{thm:GraphProdFWstraight} provide a recipe to produce, given an integer $n \geq 1$, a group of the form
$$\Gamma (\mathbb{Z}/q \mathbb{Z}) \rtimes \mathrm{Isom}(\Gamma)$$
satisfying $(\mathrm{FW}_n)$, where $\mathrm{Isom}(\Gamma)$ acts on the graph product $\Gamma (\mathbb{Z}/q \mathbb{Z})$ by permuting the vertex-groups according to its action on $\Gamma$. We refer to such semidirect products as \emph{virtual graph products}. As concrete examples:
\begin{itemize}
	\item Let $\Gamma$ denote the join of $n$ copies of the graph having only two isolated vertices and let $q \geq 2$ be an integer with no divisor in $[2,n]$. The group $\Gamma(\mathbb{Z}/q\mathbb{Z}) \rtimes \mathrm{Isom}(\Gamma)$ admits a proper and cocompact action on a median graph of cubical dimension $n$ but satisfies $(\mathrm{FW}_{n-1})$. See Corollary~\ref{cor:CubulableButFW}.
	\item Let $r,s,n \geq 1$ satisfy $r \geq s \geq n$ and let $\Gamma$ denote the graph obtained from an $r$-cube by connecting with an edge any two vertices at distance $\leq s$. For every integer $q \geq 2$ with no divisor in $[2,n]$, $\Gamma(\mathbb{Z}/q\mathbb{Z}) \rtimes \mathrm{Isom}(\Gamma)$ is cocompactly cubulable but satisfies $(\mathrm{FW}_n)$. See Theorem~\ref{thm:GraphProdFWstraight} (and Corollary~\ref{cor:FWPlus}). 
\end{itemize}
It is worth mentioning that graph products themselves usually do not have good fixed-point properties. Indeed, since they split as amalgamated products, they usually admit actions on trees without global fixed points. Here, we show that, by twisting the group just a little bit, interesting fixed-point properties appear.

\paragraph{Acknowledgements.} I am grateful to Yves Cornulier for having pointed out \cite{MR2276608} to me, to J\'er\'emie Brieussel for an interesting discussion related to \cite{MR3330165}, and to an anonymous referee for having pointed out \cite{MR2999128} to me.

\section{Fixed-point properties for virtually abelian groups}

\noindent
This section is dedicated to the proof of the following criterion:

\begin{thm}\label{thm:VirtuallyAbelian}
Let $G$ be a finitely generated group without non-abelian free subgroup. For every $n \geq 1$, the following assertions are equivalent:
\begin{itemize}
	\item[(i)] $G$ satisfies $(\mathrm{FW}_n)$;
	\item[(ii)] $G$ contains a subgroup of index $\leq n$ that admits a morphism to $\mathbb{D}_\infty$ with infinite image;
	\item[(iii)] there exists a morphism $G \to \mathbb{D}_\infty^n \rtimes \mathrm{Sym}(n)$ with infinite image.
\end{itemize}
\end{thm}

\noindent
In Section~\ref{section:DirectionalComp}, we recall basic definitions and properties related to Roller boundaries of median graphs. We also introduce \emph{directional components}, which will play an important role in the proof of Theorem~\ref{thm:VirtuallyAbelian}, whose proof is contained in Section~\ref{section:ProofVAb}. Finally, we record in Section~\ref{section:FWamApp} a few examples of groups satisfying $(\mathrm{FW}_n)$ but not $(\mathrm{FW}_{n+1})$, obtained thanks to Theorem~\ref{thm:VirtuallyAbelian}.

\subsection{Roller boundary and directional components}\label{section:DirectionalComp}

\noindent
Let $X$ be a median graph. An \emph{orientation} $\sigma$ is a map that associates to every hyperplane one of the two halfspaces it delimits and that satisfies $\sigma(J_1) \cap \sigma(J_2) \neq \emptyset$ for all hyperplanes $J_1$ and $J_2$. The orientation is \emph{principal} if there exists a vertex $x \in X$ such that $\sigma$ sends every hyperplane to the halfspace containing $x$. The \emph{Roller completion} $\bar{X}$ of $X$ is the graph whose vertices are the orientations and whose edges connect two orientations whenever they differ on a single hyperplane. The map
$$\left\{ \begin{array}{ccc} X & \to & \bar{X} \\ x & \mapsto & \text{principal orientation given by } x \end{array} \right.$$
is a graph embedding whose image is a whole connected component. The \emph{Roller boundary} $\mathfrak{R}X$ of $X$ is the complement of $X$ in $\bar{X}$. 

\medskip \noindent
Given a connected component $Y$ of $\mathfrak{R}X$ and a basepoint $x \in X$, the \emph{horomorphism} $\mathfrak{h}_Y : \mathrm{stab}(Y) \to \mathbb{Z}$ is
$$\mathfrak{h} : g \mapsto |\mathcal{W}(x|Y) \backslash \mathcal{W}(gx|Y)| - | \mathcal{W}(gx|Y) \backslash \mathcal{W}(x|Y)|,$$
where $\mathcal{W}(\cdot|\cdot)$ denotes the set of hyperplanes separating two given subsets of orientations (in the sense that our orientations take distinct values on these hyperplanes). This is is a morphism, and it does not depend on the choice of $x \in X$. We refer to \cite{CornulierCommensurated, Book} for more details on horomorphisms.

\medskip \noindent
In the rest of this preliminary section, we introduce and study a specific family of components in the Roller boundary. They will play a central role in our proof of Theorem~\ref{thm:VirtuallyAbelian}. 

\begin{definition}
Let $X$ be a median graph. A \emph{direction} in $X$ is a decreasing sequence of halfspaces. Two directions $\mathscr{A}:=(A_i)_{i \geq 0}$ and $\mathscr{B}:=(B_i)_{i \geq 0}$ are 
\begin{itemize}
	\item \emph{independent} if $A_i \cap B_j = \emptyset$ for some $i, j \geq 1$;
	\item \emph{transverse} if there exists some $N \geq 1$ such that $A_i$ and $B_j$ are transverse for all $i, j \geq N$;
	\item \emph{nested}, denoted by $\mathscr{A} \sqsubset \mathscr{B}$, if, for every $i \geq 1$, there exists some $j \geq 1$ such that
$A_j \subset B_i$;
	\item \emph{equivalent} if $\mathscr{A} \sqsubset \mathscr{B}$ and $\mathscr{B} \sqsubset \mathscr{A}$. 
\end{itemize}
\end{definition}

\noindent
Given a median graph $X$, a direction $\mathscr{D}:= (D_i)_{i \geq 0}$, and a vertex $x \in V(X)$, we define the push of $x$ towards $\mathscr{D}$ as
$$\mathscr{D}x : J \mapsto  \left\{ \begin{array}{l} \text{halfspace containing some $D_i$ when possible}\\ \text{halfspace containing $x$ otherwise} \end{array} \right..$$
Observe that $\mathscr{D}x$ is an orientation. Indeed, otherwise we can find two hyperplanes $A$ and $B$ such that $\mathscr{D}x(A) \cap \mathscr{D}x(B)= \emptyset$. By construction, no $D_i$ can be contained in either $\mathscr{D}x(A)$ or $\mathscr{D}x(B)$, hence $x \in \mathscr{D}x(A) \cap \mathscr{D}x(B)= \emptyset$, a contradiction. 

\medskip \noindent
The key observation is that pushing all the vertices of a median graph towards a given direction yields an entire component of the Roller boundary. 

\begin{prop}\label{prop:DxComponent}
Let $X$ be a median graph and $\mathscr{D}:=(D_i)_{i \geq 0}$ a direction. Then,
$$\{ \mathscr{D}x \mid x \in V(X)\}$$
is the vertex-set of a connected component of $\mathfrak{R}X$. 
\end{prop}

\noindent
The following elementary observation of median geometry will be necessary in our proof of the proposition:

\begin{lemma}\label{lem:SepIntHalfspaces}
Let $X$ be a median graph and $D_1, \ldots, D_n$ pairwise intersecting halfspaces. Given a vertex $x \in V(X)$, if a hyperplane $J$ separates $x$ from $D_1 \cap \cdots \cap D_n$, then $J$ separates $x$ from some $D_i$. 
\end{lemma}

\begin{proof}
Let $J^+$ denote the halfspace delimited by $J$ that contains $x$. If $J$ does not separate $x$ from $D_i$ for every $i$, then $J^+ \cap D_i \neq \emptyset$ for every $i$. Then, it follows from the Helly property satisfied by convex subgraphs in median graphs that $J^+$ intersects $D_1 \cap \cdots \cap D_n$. This implies that $J$ does not separate $x$ from $D_1 \cap \cdots \cap D_n$. 
\end{proof}

\begin{proof}[Proof of Proposition~\ref{prop:DxComponent}.]
First of all, observe that pushing a vertex of $X$ towards $\mathscr{D}$ always leads to the same component of $\mathfrak{R}X$. This a consequence of the following observation:

\begin{claim}\label{claim:DxDiffHyp}
For all vertices $x,y \in V(X)$, $\mathscr{D}x$ and $\mathscr{D}y$ may only differ on hyperplanes separating $x$ and $y$. Consequently, $\mathscr{D}x$ and $\mathscr{D}y$ differ on only finitely many hyperplanes.
\end{claim}

\noindent
Let $J$ be a hyperplane on which $\mathscr{D}x$ and $\mathscr{D}y$ differ. No $D_i$ can be contained in a halfspace delimited by $J$, since otherwise $\mathscr{D}x(J)$ and $\mathscr{D}y(J)$ would be both this halfspace. Consequently, $\mathscr{D}x(J)$ (resp.\ $\mathscr{D}y(J)$) is the halfspace delimited by $J$ that contains $x$ (resp.\ $y$). It follows that $J$ separates $x$ and $y$, concluding the proof of Claim~\ref{claim:DxDiffHyp}.

\medskip \noindent
So far, we have proved that $\{\mathscr{D}x \mid x \in V(X)\}$ is contained in a single component of $\mathfrak{R}X$. It remains to verify that every vertex of this component is a $\mathscr{D}x$ for some vertex $x \in V(X)$. In other words, given an orientation $\sigma$ that differ on only finitely many hyperplanes with some (or equivalently, any) $\mathscr{D}x$, we want to prove that $\sigma = \mathscr{D}y$ for some $y \in V(X)$. 

\medskip \noindent
We claim that the vertex $y$ we are looking for is the projection of $x$ to 
$$C:= \bigcap\limits_{\sigma \text{ and } \mathscr{D}x \text{ differ on } J} \sigma(J).$$
Notice that $C$ is non-empty as a consequence of the Helly property satisfied by convex subgraphs in median graphs. We start by observation that:

\begin{claim}\label{claim:HypSepXandY}
Every hyperplane on which $\sigma$ and $\mathscr{D}x$ disagree separates $x$ and $y$.
\end{claim}

\noindent
Let $J$ be a hyperplane on which $\sigma$ and $\mathscr{D}x$. It is clear that the hyperplanes on which $\sigma$ and $\mathscr{D}X$ disagree separate $x$ from $C$, and a fortiori from $y$. Therefore, $J$ must separate $x$ and $y$, proving Claim~\ref{claim:HypSepXandY}. 

\medskip \noindent
Now, we are ready to prove that $\sigma$ coincides with $\mathscr{D}y$. In other words, given an arbitrary hyperplane $J$, we want to prove that $\sigma(J)= \mathscr{D}y(J)$. 

\medskip \noindent
If $J$ does not separate $x$ and $y$, then we know from Claim~\ref{claim:HypSepXandY} that $\sigma$ and $\mathscr{D}x$ agree at $J$. Moreover, it follows from Claim~\ref{claim:DxDiffHyp} that $\mathscr{D}x$ and $\mathscr{D}y$ also agree at $J$. Hence $\sigma(J)= \mathscr{D}x(J)= \mathscr{D}y(J)$, as desired. 

\medskip \noindent
Next, assume that $J$ does separate $x$ and $y$. Because $y$ is the projection of $x$ to $C$, necessarily $J$ separates $x$ from $C$. According to Lemma~\ref{lem:SepIntHalfspaces}, $J$ separates $x$ from some hyperplane $H$ on which $\sigma$ and $\mathscr{D}x$ disagree. Let $J^-$ and $H^-$ (resp.\ $J^+$ and $H^+$) denote the halfspaces respectively delimited by $J$ and $H$ that contain $C$ (resp.\ do not contain $C$). By definition of $C$, we must have $\sigma(H)=H^+$, which implies that $\mathscr{D}x(H)=H^-$. Since $\sigma(H)=H^+ \subset J^+$, necessarily $\sigma(J)=J^+$. Since $y$ belongs to $J^+$, the only possibility for $\mathscr{D}y(J)$ to be equal to $J^-$ is that some $D_i$ is contained in $J^-$. If so, $\sigma$ and $\mathscr{D}x$ must disagree on every hyperplane bounding $D_j$ for $j \geq i$, which is impossible since $\sigma$ and $\mathscr{D}x$ disagree on only finitely many hyperplanes. Thus, we must have $\mathscr{D}y(J)=J^+= \sigma(J)$, as desired. 
\end{proof}

\noindent
Proposition~\ref{prop:DxComponent} allows us to define:

\begin{definition}
Let $X$ be a median graph. A component $Y$ of $\mathfrak{R}X$ is \emph{directional} if there exists a direction $\mathscr{D}$ such that $\{ \mathscr{D}x \mid x \in V(X)\}$ is the vertex-set of $Y$. 
\end{definition}

\noindent
According to Proposition~\ref{prop:DxComponent}, every direction defines a directional component. Moreover, it is clear that:

\begin{fact}\label{fact:DirectionEqui}
Two equivalent directions define the same direction components. \qed
\end{fact}

\noindent
We conclude this section with three preliminary statements related to directional components, namely Corollary~\ref{cor:FinitelyMinComp} and Lemmas~\ref{lem:DirMin} and~\ref{lem:HoroNotTrivial}. Corollary~\ref{cor:FinitelyMinComp} will be a straightforward consequence of the following observation:

\begin{prop}\label{prop:FiniteDirections}
Let $X$ be a median graph. If $\mathscr{A}_1, \ldots, \mathscr{A}_k$ is a collection of directions that pairwise are neither equivalent nor independent, then $k \leq \dim_\square(X)$.
\end{prop}

\noindent
Our proposition is an easy consequence of our next two lemmas.

\begin{lemma}\label{lem:NestedDichot}
Let $X$ be a median graph and $\mathscr{A}, \mathscr{B}$ two directions. If $\mathscr{A}$ and $\mathscr{B}$ are neither independent nor transverse, then $\mathscr{A} \sqsubset \mathscr{B}$ or $\mathscr{B} \sqsubset \mathscr{A}$.
\end{lemma}

\begin{proof}
Write $\mathscr{A}$ as $(A_i)_{i\geq 1}$ and $\mathscr{B}$ as $(B_i)_{i\geq 1}$. Because $\mathscr{A}$ and $\mathscr{B}$ are neither independent nor transverse, we know that, for every $n \geq 1$, there exist $i(n), j(n) \geq n$ such that
$A_{i(n)} \subset B_{j(n)}$ or $A_{i(n)} \subset B_{j(n)}$. If the first inclusion occurs for infinitely many indices, then, given a $k \geq 1$, there exists some $n \geq 1$ such that $i(n), j(n) \geq k$ and $A_{i(n)} \subset B_{j(n)}$, hence $B_k \supset B_{j(n)} \supset A_{i(n)}$. Thus, we have $\mathscr{A} \sqsubset \mathscr{B}$. Similarly, if the second inclusion occurs for infinitely many indices, then $\mathscr{B} \sqsubset \mathscr{A}$.
\end{proof}

\begin{lemma}\label{lem:NestedTrans}
Let $X$ be a median graph and $\mathscr{A} = (A_i)_{i\geq 1}, \mathscr{B} = (B_i)_{i\geq 1}$ two directions. If $\mathscr{A},\mathscr{B}$ are not equivalent but $\mathscr{A} \sqsubset \mathscr{B}$, then there exists some $p \geq 1$ such that, for every $i \geq p$, there exists some $q \geq 1$ such that $A_i$ is transverse to $B_j$ for every $j \geq q$.
\end{lemma}

\noindent
The conclusion should be read as follows: up to removing finitely many halfspaces from $\mathscr{A}$, every halfspace in $\mathscr{A}$ is transverse to all but finitely many halfspaces in $\mathscr{B}$.

\begin{proof}[Proof of Lemma~\ref{lem:NestedTrans}.]
Because $\mathscr{A} \sqsubset \mathscr{B}$, we know that, for every $n \geq 1$, there exists some $j(n) \geq 1$ such that $B_n \supset A_{j(n)}$. And, because we do not have $\mathscr{B} \sqsubset \mathscr{A}$, we know that there exists some $p \geq 1$ such that $A_p$ does not contain any $B_j$. Fix some $i \geq p$. Up to removing finitely many halfspaces from $\mathscr{B}$, we assume without loss of generality that $j(n) \geq i$ for every $n \geq 1$. Let $n$ the smallest index such that $B_n$ is not transverse to $A_i$. Then, for every $m \geq 1$ such that $B_m$ is not transverse to $A_i$, we must have $A_i \subset B_m \subset B_n$. Because there can exist only finitely many such intermediate halfspaces, it follows that
$B_m$ is transverse to $A_i$ for all but finitely many $m$.
\end{proof}

\begin{proof}[Proof of Proposition~\ref{prop:FiniteDirections}.]
As a consequence of Lemma~\ref{lem:NestedDichot}, for all distinct $1 \leq i,j \leq n$, $\mathscr{A}_i$ and $\mathscr{A}_j$ are either transverse or nested. We deduce from Lemma~\ref{lem:NestedTrans} that we can find halfspaces $D_1 \in \mathscr{A}_1, \ldots , D_k \in \mathscr{A}_k$ that are pairwise transverse, hence $k \leq \dim_\square(X)$ as desired.
\end{proof}

\noindent
We are now ready to state and prove our first preliminary statement related to directional components. In the sequel, given a median graph $X$ and two subsets $R,S \subset \mathfrak{R}X$, 
\begin{itemize}
	\item we say that $\mathscr{D}:=(D_i)_{i \geq 0}$ is a direction to $R$ if $\sigma(\text{hyperplane bounding } D_i)=D_i$ for every $i \geq 0$;
	\item we denote by $R \leq S$ if every direction to $R$ is a direction to $S$. 
\end{itemize}

\begin{cor}\label{cor:FinitelyMinComp}
Let $X$ be a median graph of finite cubical dimension and $S \subset \mathfrak{R}X$ a non-empty subset. There are at most $\dim_\square(X)$ directional components $Y \subset \mathfrak{R}X$ satisfying $Y \leq S$. 
\end{cor}

\begin{proof}
Let $Y_1, \ldots, Y_n$ be pairwise distinct directional components satisfying $Y_1, \ldots, Y_n \leq S$. For every $1 \leq i \leq n$, let $\mathscr{D}_i$ be a direction defining $Y_i$. It follows from Fact~\ref{fact:DirectionEqui} that $\mathscr{D}_1,\ldots, \mathscr{D}_n$ are pairwise non-equivalent. Moreover, since $Y_i \leq S$ for every $1 \leq i \leq n$, necessarily no two $Y_i$ can be independent. Thus, Lemma~\ref{prop:FiniteDirections} implies that $n \leq \dim_\square(X)$. 
\end{proof}

\noindent
We now prove our last two preliminary statements related to directional components.

\begin{lemma}\label{lem:DirMin}
Let $X$ be a median graph and $\mathscr{D}:=(D_i)_{i \geq 0}$ a direction. Let $Y$ denote the directional component defined by $\mathscr{D}$. Given a subset $S \subset \mathfrak{R}X$, if $S \subset D_i$ for every $i \geq 0$, then $Y \leq S$. 
\end{lemma}

\begin{proof}
Let $\mathscr{A}$ be an arbitrary direction to $Y$. Given a halfspace $A \in \mathscr{A}$, we want to prove that $S \subset A$. If there exists some $i \geq 0$ such that $D_i \subset A$, then clearly $S \subset A$. Otherwise, assume that $D_i \nsubseteq A$ for every $i \geq 0$. Fix a vertex $x \in V(X)$ notin $A$. Then, if $J$ denotes the hyperplanes bounding $A$, we have $\mathscr{D}x(J) \neq A$. Thus, on the one hand, because $Y \subset A$ as $\mathscr{A}$ is a direction to $Y$, $Y$ is disjoint from $\mathscr{D}x(J)$; on the other hand, $\mathscr{D}x(J)$ must belong to $Y$ since $\mathscr{D}$ defines $Y$. A contradiction. 
\end{proof}

\begin{lemma}\label{lem:HoroNotTrivial}
Let $X$ be a median graph of finite cubical dimension, $\mathscr{D}:= (D_i)_{i \geq 0}$ a direction, and $g \in \mathrm{Isom}(X)$ a loxodromic isometry. Let $Y$ denote the directional component defined by $\mathscr{D}$. If $g$ stabilises $Y$ and has an axis crossing each $D_i$, then $\mathfrak{h}_Y(g) = \pm \|g\|$. 
\end{lemma}

\begin{proof}
Fix an axis $\gamma$ of $g$ that crosses each $D_i$ and a vertex $o$ on $\gamma$. We have
$$\pm \mathfrak{h}_Y(g)= |\mathcal{W}(o|Y) \backslash \mathcal{W}(go|Y)| - |\mathcal{W}(go|Y)\backslash \mathcal{W}(o|Y)|.$$
We claim that $\mathcal{W}(go|Y) \backslash \mathcal{W}(o|Y)=\emptyset$. Since
$$\mathcal{W}(o|go) = \mathcal{W}(o|Y) \backslash \mathcal{W}(go|Y) \sqcup \mathcal{W}(go|Y)\backslash \mathcal{W}(o|Y),$$
this will prove that $\pm \mathfrak{h}_Y(g)= |\mathcal{W}(o|go)| = d(o,go) = \|g\|$, as desired. 

\medskip \noindent
Assume for contradiction that there exists some hyperplane $J \in \mathcal{W}(go|Y) \backslash \mathcal{W}(o|Y)$. Let $D$ denote the halfspace delimited by $J$ that contains $go$. Because $X$ has finite cubical dimension, we can find a $\langle g \rangle$-translate of $J$, say $g^NJ$, that is not transverse to $J$ and such that $g^ND$ contains $o$. Notice that, since $D$ is disjoint from $Y$ and because $Y$ is stabilised by $g$, $g^ND$ must be disjoint from $Y$. Since $\mathscr{D}o \in Y$, this implies that $\mathscr{D}o(J) \neq g^nD$. As $o$ belongs to $g^N D$, necessarily some $D_i$ must be contained in the complement of $g^ND$, which is impossible since $g^ND \cap D_i$ contains a half-ray of $\gamma$ for every $i \geq 0$. 
\end{proof}

\subsection{Proof of the criterion}\label{section:ProofVAb}

\noindent
In this section, our goal is to prove Theorem~\ref{thm:VirtuallyAbelian}. We start by stating and proving a few preliminary observations. First of all:

\begin{fact}\label{fact:VirtuallyFW}
If a group $G$ contains a subgroup of index $k$ that does not satisfy $(\mathrm{FW}_n)$, then $G$ does not satisfy $(\mathrm{FW}_{nk})$. 
\end{fact}

\noindent
Indeed, it is well known that, if a subgroup $H \leq G$ of index $k$ admits an action on a median graph $X$ (or, more generally, on an arbitrary metric space), then $G$ naturally acts on $X^{[G:H]}$. This proves Fact~\ref{fact:VirtuallyFW}.

\medskip \noindent
Next, we prove the equivalence between (ii) and (iii) in Theorem~\ref{thm:VirtuallyAbelian}, namely:

\begin{lemma}\label{lem:FWasMorph}
A group $G$ contains a subgroup of index $\leq n$ admitting a morphism to $\mathbb{D}_\infty$ with infinite image if and only if there exists a morphism $G \to \mathbb{D}_\infty^n \rtimes \mathrm{Sym}(n)$ with infinite image.
\end{lemma}

\begin{proof}
Assume that there exists a morphism $\varphi : G \to \mathbb{D}_\infty^n \rtimes \mathrm{Sym}(n)$ with infinite image. Thus, there exists some $g \in G$ such that $\varphi(g)=(g_1, \ldots, g_n,\sigma)$ with $g_i$ infinite-order for some $1 \leq i \leq n$. Let $F_i$ denote the index-$n$ subgroup $\mathrm{D}_\infty^n \rtimes \mathrm{Fix}(i)$ of $\mathbb{D}_\infty^n \rtimes \mathrm{Sym}(n)$. Notice that $F_i$ is naturally isomorphic to $\mathbb{D}_\infty \times ( \mathbb{D}_\infty^{n-1} \rtimes \mathrm{Sym}(n-1))$, hence a surjective morphism $\psi : F_i \twoheadrightarrow \mathbb{D}_\infty$. Now, $\varphi^{-1}(F_i)$ defines a subgroup of $G$ of index $\leq n$ whose image under $\psi \circ \varphi$ is infinite as $\varphi(g^{n!})=(g_1^{n!}, \ldots, g_n^{n!},\mathrm{id})$ belongs to $F_i$ and $\psi( \varphi(g^{n!}))=g_i^{n!}$ has infinite order in $\mathbb{D}_\infty$.

\medskip \noindent
Conversely, assume that $G$ contains a subroup $H$ of index $\leq n$ admitting a morphism to $\mathbb{D}_\infty$ with infinite image. The last condition amounts to saying that $H$ acts on the graph $\mathbb{Z}$ with unbounded orbits. It is well-known that such an action of $H$ can be promoted to an action of $G$ on the graph $\mathbb{Z}^{G/H}$ with unbounded orbits. Since the isometry group of $\mathbb{Z}^{G/H}$ is isomorphic to $\mathbb{D}_\infty^{[G:H]} \rtimes \mathrm{Sym}([G:H])$, which embeds into $\mathbb{D}_\infty^n \rtimes \mathrm{Sym}(n)$ as $n \geq [G:H]$, the desired conclusion follows. 
\end{proof}

\noindent
Our next preliminary lemma shows how to extend a morphism to $\mathbb{Z}$ defined on an index-$2$ subgroup to a morphism to the infinite dihedral group $\mathbb{D}_\infty$ on the whole group. 

\begin{lemma}\label{lem:MorphExtension}
Let $G$ be a group, $H \leq G$ an index-$2$ subgroup, and $\varphi : H \to \mathbb{Z}$ a morphism. If there exists some $a \in G \backslash H$ such that $\varphi(a^2)=0$ and such that $\varphi(aha^{-1})=- \varphi(h)$ for every $h \in H$, then $\varphi$ extends to a morphism $G \to \mathbb{Z} \rtimes \mathbb{Z}_2 \simeq \mathbb{D}_\infty$. 
\end{lemma}

\begin{proof}
Since $G= H \sqcup Ha$, we can define the map
$$\psi : \left\{ \begin{array}{ccc} G & \to & \mathbb{Z} \rtimes \mathbb{Z}_2 \\ h & \mapsto & (\varphi(h),1) \\ ha & \mapsto & (\varphi(h),-1) \end{array} \right.$$
where $\mathbb{Z}_2$ is thought of as the multiplicative group $\{\pm 1\}$. Clearly, $\psi$ extends $\varphi$. It remains to verify that $\psi$ is a morphism. For all $h_1,h_2 \in H$, we have
$$\psi(h_1h_2) = (\varphi(h_1h_2),1) = (\varphi(h_1),1)(\varphi(h_2),1) = \psi(h_1)\psi(h_2);$$
and
$$\psi(h_1 \cdot h_2a) = (\varphi(h_1h_2),-1) = (\varphi(h_1),1)(\varphi(h_2), -1)= \psi(h_1)\psi(h_2a);$$
and
$$\begin{array}{lcl} \psi(h_1a \cdot h_2) & = & \psi(h_1ah_2a^{-1}a) = (\varphi(h_1ah_2a^{-1}),-1) = (\varphi(h_1)-\varphi(h_2), -1) \\ \\ & = & (\varphi(h_1),-1) (\varphi(h_2),1) = \psi(h_1a) \psi(h_2)); \end{array}$$
and finally
$$\begin{array}{lcl} \psi(h_1a \cdot h_2a) & = & \psi ( h_1ah_2a^{-1}a^2 ) = (\varphi(h_1ah_2a^{-1}a^2), 1) = (\varphi(h_1)- \varphi(h_2), 1) \\ \\ & = & (\varphi(h_1),-1)(\varphi(h_2),-1) = \psi(h_1a) \psi(h_2a). \end{array}$$
We conclude, as desired, that $\psi$ defines a morphism $G \to \mathbb{D}_\infty$ that extends $\varphi$. 
\end{proof}

\noindent
Typically, Lemma~\ref{lem:MorphExtension} will be useful in the following situation. If a group $G$ acts on a median graph $X$ and stabilises a pair $\{\alpha,\omega\}$ of vertices in the Roller boundary, then $G$ contains a subgroup $H$ of index $\leq 2$ that fixes $\alpha$. Thus, the horomorphism $\mathfrak{h}_\alpha$ induces a morphism $H \to \mathbb{Z}$. Lemma~\ref{lem:MorphExtension} will allow us to promote this morphism to a morphism $G \to \mathbb{D}_\infty$. In order to apply Lemma~\ref{lem:MorphExtension}, the following observation will be necessary:

\begin{lemma}\label{lem:Switch}
Let $X$ be a median and $g \in \mathrm{Isom}(X)$ a loxodromic isometry. Assume that there exist two vertices $\alpha,\omega \in \mathfrak{R}X$ such that $I(\alpha,\omega) \cap X \neq \emptyset$ and such that $g$ switches $\alpha$ and $\omega$. Then $\mathfrak{h}_\alpha(g^2)=0$.
\end{lemma}

\begin{proof}
Up to replacing $X$ with $X \cap I(\alpha,\omega)$, we can assume for convenience that $X=X \cap I(\alpha,\omega)$, which amounts to saying that every hyperplane of $X$ separates $\alpha$ and $\omega$. Fix a vertex $o$ on an axis of $g$. By definition of $\mathfrak{h}_\alpha$, we have
$$\begin{array}{lcl} \mathfrak{h}_\alpha(g^2) & = & |\mathcal{W}(o|\alpha)\backslash \mathcal{W}(g^2o|\alpha)| - |\mathcal{W}(g^2o|\alpha)\backslash \mathcal{W}(o|\alpha)| \\ \\ & = & |\mathcal{W}(o|g^2o,\alpha)| - |\mathcal{W}(g^2o| o , \alpha)| \end{array}$$
On the one hand, 
$$\begin{array}{lcl} \mathcal{W}(o|g^2o,\alpha) & = & \mathcal{W}(o, \omega|g^2o,\alpha) = \mathcal{W}(o,go,\omega|g^2o,\alpha) \sqcup \mathcal{W}(o,\omega|go,g^2o,\alpha) \\ \\ &= &  \mathcal{W}(go,\omega|g^2o,\alpha) \sqcup \mathcal{W}(o,\omega|go,\alpha) \\ \\ & = & g \mathcal{W}(o,\alpha | go,\omega) \sqcup \mathcal{W}(o,\omega | go, \alpha) \end{array}$$
and
$$\begin{array}{lcl} \mathcal{W}(g^2o|o,\alpha) & = & \mathcal{W}(g^2o,\omega|o,\alpha) =  \mathcal{W}(g^2o,go,\omega |o,\alpha) \sqcup \mathcal{W}(g^2o, \omega|o,go,\alpha) \\ \\ & = &  \mathcal{W}(go,\omega |o,\alpha) \sqcup \mathcal{W}(g^2o, \omega|go,\alpha) \\ \\ & = &  \mathcal{W}(go,\omega |o,\alpha) \sqcup g \mathcal{W}(go, \alpha | o,\omega) \end{array}$$
We conclude that $\mathfrak{h}_\alpha(g^2)=0$, as desired.
\end{proof}

\begin{proof}[Proof of Theorem~\ref{thm:VirtuallyAbelian}.]
Assume that $G$ contains a subgroup $H$ of index $\leq n$ that admits a morphism to $\mathbb{D}_\infty$ with infinite image. Then, $H$ acts on the bi-infinite line $\mathbb{E}^1$ with unbounded orbits, which can be promoted to an action of $G$ on $\mathbb{E}^{[G:H]}$ with unbounded orbits according to Fact~\ref{fact:VirtuallyFW}. Thus, $G$ does not satisfy $(\mathrm{FW}_n)$.

\medskip \noindent
Conversely, assume that $G$ does not satisfy $(\mathrm{FW}_n)$. Let $X$ be a median graph of cubical dimension $\leq n$ on which $G$ acts with unbounded orbits. Up to replacing $X$ with its cubical subdivision, we assume that $G$ acts without hyperplane-inversion. Because $G$ has no free subgroup of rank $\geq 2$, it must have a finite orbit in the Roller boundary of $X$ \cite{MR2827012} (see \cite[Theorem~2.10]{MR4062290} for an explicit statement). Since $X \cup \mathfrak{R}X$ is a median algebra (of rank $\leq n$), it follows from \cite[Lemma~21.1.1]{BowMedian} that $G$ stabilises a median cube $Q \subset \mathfrak{R}X$, necessarily of dimension $\leq \dim_\square(X)$. 

\medskip \noindent
Because $G$ is finitely generated and acts on $X$ with unbounded orbits, we know from \cite{MR1347406} (see \cite{MR4162939} for details) that $G$ contains a loxodromic element, say $g \in G$. We fix a hyperplane $J$ that crosses an axis of $g$. Up to replacing $g$ with a non-trivial power, we can assume that the $g^nJ$, $n \in \mathbb{Z}$ are pairwise non-transverse. In other words, $\{g^nJ \mid n \in \mathbb{Z}\}$ defines a bi-infinite chain of hyperplanes. We denote by $J^-$ and $J^+$ the halfspaces delimited by $J$ so that $gJ^+ \subset J^+$. We set
$$\mathfrak{R}_\pm(g):= \{ \sigma \in \mathfrak{R}X \mid \sigma(g^nJ)= g^nJ^\pm \text{ for every } n \in \mathbb{Z}\}$$
and $\mathfrak{R}(g):= \mathfrak{R}_-(g) \sqcup \mathfrak{R}_+(g)$. 

\begin{claim}\label{claim:QinRg}
We have $Q \subset \mathfrak{R}(g)$. 
\end{claim}

\noindent
It suffices to verify that, if $\sigma \in \mathfrak{R}X \backslash \mathfrak{R}(g)$, then the orbit $\langle g \rangle \cdot x$ is infinite. Since $\sigma \notin \mathfrak{R}(g)$, we can find $n \in \mathbb{Z}$ such that $\sigma(g^kJ)=g^kJ^+$ for every $k \leq n$ and $\sigma(g^{k}J)=g^{k}J^-$ for every $k \geq n+1$. In other words, $\sigma$ lies between $g^nJ$ and $g^{n+1}J$. Then, for every $i \in \mathbb{Z}$, $g^i\sigma$ lies between $g^{n+i}J$ and $g^{n+i+1}J$, proving that the $g^i \sigma$ are pairwise distinct. Thus, Claim~\ref{claim:QinRg} is proved. 

\medskip \noindent
Now, we distinguish two cases, depending on how $Q$ sits inside $\mathfrak{R}(g)$.

\medskip \noindent
First, assume that $Q \subset \mathfrak{R}_+(g)$ or $\mathfrak{R}_-(g)$. Up to replacing $g$ with $g^{-1}$, we can assume that $Q \subset \mathfrak{R}_+(g)$. Let $Y$ denote the directional component given by $\{g^nJ^+ \mid n \geq 0\}$. We know from Lemma~\ref{lem:DirMin} that $Y \leq Q$, so we deduce from Corollary~\ref{cor:FinitelyMinComp} that $G$ contains a subgroup $H$ of index $\leq n$ stabilising $Y$. Up to replacing $g$ with a non-trivial power that belongs to $H$, it follows from Lemma~\ref{lem:HoroNotTrivial} that $\mathfrak{h}_Y(g) \neq 0$. Thus, $H$ surjects onto $\mathbb{Z}$ through the horomorphism $\mathfrak{h}_Y$. 

\medskip \noindent
Next, assume that $Q$ has a non-empty intersection with both $\mathfrak{R}_-(g)$ and $\mathfrak{R}_+(g)$. Because $\mathfrak{R}_-(g)$ and $\mathfrak{R}_+(g)$ are both convex in the median algebra $\overline{X}$, the partition $Q = (Q \cap \mathfrak{R}_-(g)) \sqcup (Q \cap \mathfrak{R}_+(g))$ must correspond to two complementary halfspaces. In other words, $Q$ contains a codimension-one subcube $Q_-$ in $\mathfrak{R}_-(g)$ with its opposite codimension-one subcube $Q_+$ in $\mathfrak{R}_+(g)$. Since $Q$ contains $\leq n$ pairs of opposite facets, $G$ must contain a subgroup $H$ of index $\leq n$ that stabilises the pair $\{Q_-,Q_+\}$. Our goal is to prove that there exists a morphism $H \to \mathbb{D}_\infty$ with infinite image. For this, we want to apply Lemma~\ref{lem:MorphExtension} to the subgroup $H_0 \leq H$ of index $\leq 2$ that stabilises both $Q_-$ and $Q_+$. 

\medskip \noindent
Let $Z$ be the median graph obtained by cubulating the wallspace $(X, \mathcal{W}(Q_-|Q_+))$. In $Z$, the cubes $Q_-$ and $Q_+$ becomes respectively two vertices $\xi_-$ and $\xi_+$ in $\mathfrak{R}X$; $g$ remains loxodromic; and $\xi_- \in \mathfrak{R}_-(g)$, $\xi_+ \in \mathfrak{R}_+(g)$. Let $\mathfrak{h}_\pm$ denote the horomorphism given by $\xi_\pm$. Notice that $g \in H_0$ and $\mathfrak{h}_\pm(g) \neq 0$. Moreover, $\mathfrak{h}_-=-\mathfrak{h}_+$ on $H_0$. Indeed, since we know, by construction of $Z$, that all the hyperplanes in $Z$ separate $\xi_-$ and $\xi_+$, being separated from $\xi_-$ by some hyperplane $J$ amounts to not being separated from $\xi_+$ by $J$, which implies that 
$$\begin{array}{lcl} \mathfrak{h}_-(h) & = & |\mathcal{W}(o|\xi_-) \backslash \mathcal{W}(ho|\xi_-)| - | \mathcal{W}(ho|\xi_-)\backslash \mathcal{W}(o|\xi_-)| \\ \\ & = & |\mathcal{W}(ho|\xi_+) \backslash \mathcal{W}(o|\xi_-)| -  |\mathcal{W}(o|\xi_+) \backslash \mathcal{W}(ho|\xi_+)| \\ \\ &= & -\mathfrak{h}_+(h), \end{array}$$
where $o \in V(X)$ is a basepoint and $h \in H_0$ an arbitrary element. 

\medskip \noindent
If $H=H_0$, then we know that $H$ surjects onto $\mathbb{Z}$ and we are done. From now on, assume that $H_0 \subsetneq H$. 

\medskip \noindent
Fix an element $a \in H \backslash H_0$. In other words, $a$ switches $Q_-$ and $Q_+$. First, notice that $\mathfrak{h}_+(a^2)=0$. Indeed, the equality is clear if $a$ fixes a vertex, and it follows from Lemma~\ref{lem:Switch} if $a$ is loxodromic. Next, let $h \in H_0$ be an arbitrary element. Given a basepoint $o \in V(X)$, we have
$$\begin{array}{lcl} \mathfrak{h}_+(aha^{-1}) & = & |\mathcal{W}(o|\xi_+) \backslash \mathcal{W}(aha^{-1}o|\xi_+)| - |\mathcal{W}(aha^{-1}o|\xi_+) \backslash \mathcal{W}(o,\xi_+)| \\ \\ & = & |\mathcal{W}(a^{-1}o|a^{-1}\xi_+) \backslash \mathcal{W}(ha^{-1}o|a^{-1}\xi_+)| - |\mathcal{W}(ha^{-1}o|a^{-1} \xi_+) \backslash \mathcal{W}(a^{-1}o|a^{-1}\xi_+)| \\ \\ &= & |\mathcal{W}(a^{-1}o|\xi_-) \backslash \mathcal{W}(ha^{-1}o|\xi_-)| - |\mathcal{W}(ha^{-1}o| \xi_-) \backslash \mathcal{W}(a^{-1}o|\xi_-)| \\ \\ & = & \mathfrak{h}_-(h) = - \mathfrak{h}_+(h)
\end{array}$$
Thus, Lemma~\ref{lem:MorphExtension} applies and shows that $\mathfrak{h}_+ : H_0 \to \mathbb{Z}$ extends to a morphism $H \to \mathbb{D}_\infty$, whose image is necessarily infinite. This concludes the proof of our theorem. 
\end{proof}

\begin{remark}
It is worth mentioning that the second paragraph of our proof provides a generalisation of \cite[Corollary (Tits' alternative)]{MR4861512}. Namely, if a group with no $\mathbb{F}_2$ acts on a median graph of finite cubical dimension, then $G$ has an orbit of size $2^N$ in $\overline{X}$ for some $N \leq \mathrm{dim}_\square(X)$. Indeed, it follows from \cite[Theorem~2.10]{MR4062290} that $G$ has a finite orbit in $\overline{X}$. Since $\overline{X}$ is a median algebra of rank $\leq n$, it follows from \cite[Lemma~21.1.1]{BowMedian} that $G$ stabilises a median cube. Then, the desired conclusion follows. 
\end{remark}

\subsection{Applications}\label{section:FWamApp}

\noindent
As a concrete application of Theorem~\ref{thm:VirtuallyAbelian}, we investigate the properties $(\mathrm{FW}_n)$ for two families of virtually abelian groups. 

\begin{cor}\label{cor:WreathFW}
For every transitive permutation group $F \curvearrowright [n]:=\{1, \ldots, n\}$, the group $\mathbb{D}_\infty \wr_{[n]} F$ satisfies $(\mathrm{FW}_{n-1})$.
\end{cor}

\noindent
Recall that, given a group $A$ and a group $B$ acting on a set $S$, the \emph{permutational wreath product} $A \wr_S B$ is the semidirect product $\left( \bigoplus_S A \right) \rtimes B$ where $B$ acts on the direct sum by permuting the coordinates according to its action on $S$. 

\medskip \noindent
It is worth mentioning that Corollary~\ref{cor:WreathFW} also follows from \cite{MR2999128}, since it can be deduced from \cite[Corollaries~3.5 or~3.6]{MR2999128} that, more generally, $\mathbb{D}_\infty \wr_{[n]} F$ satisfies the fixed-point property on complete CAT(0) spaces of dimension $<n$. 

\begin{proof}[Proof of Corollary~\ref{cor:WreathFW}.]
As a consequence of Theorem~\ref{thm:VirtuallyAbelian} and Lemma~\ref{lem:FWasMorph}, in order to justify that $G$ has $(\mathrm{FW}_{n-1})$, it suffices to show that every morphism $\varphi : \mathbb{D}_\infty \wr_{[n]} F \to \mathbb{D}_\infty^{n-1} \rtimes \mathrm{Sym}(n-1)$ has finite image. This is what we verify in the rest of the proof.

\medskip \noindent
For convenience, for every $1 \leq i \leq n$, we denote by $D_i$ the factor $\mathbb{D}_\infty$ in $\mathbb{D}_\infty \wr_{[n]} F$ indexed by $i \in [n]$. Thus, $\mathbb{D}_\infty \wr_{[n]} F$ decomposes as $(D_1 \times \cdots \times D_n) \rtimes F$. 

\medskip \noindent
Notice that, because the $D_i$ are all conjugate in $\mathbb{D}_\infty \wr_{[n]} F$, since $F$ acts transitively on $[n]$, if we know that at least of $D_i$ is finite then we know that they are all finite, which implies that $\varphi(\mathbb{D}_\infty \wr_{[n]} F)= \varphi(D_1) \cdots \varphi(D_n) \varphi(F)$ must be finite as well.

\medskip \noindent
Thus, we can assume that $\varphi(D_i)$ is infinite for every $1 \leq i \leq n$. Since $\mathbb{D}_\infty^{n-1} \rtimes \mathrm{Sym}(n-1)$ cannot contain $n$ infinite cyclic subgroups that are pairwise not commensurable (as a group virtually isomorphic to $\mathbb{Z}^{n-1}$), there must exist two distinct $1 \leq i,j \leq n$ such that $\varphi(D_i)$ and $\varphi(D_j)$ are commensurable in $\mathbb{D}_\infty^{n-1} \rtimes \mathrm{Sym}(n-1)$. In other words, if $a_i,b_i \in D_i$ (resp.\ $a_j,b_j \in D_i$) denote two generators of order two, there exist $p,q \in \mathbb{Z} \backslash \{0\}$ such that $\varphi(a_ib_i)^p = \varphi(a_jb_j)^q$. Notice that
$$\begin{array}{lcl} \displaystyle \varphi \left( \left[ a_i, (a_ib_i)^p \right] \right) & = & \displaystyle  \left[ \varphi(a_i), \varphi(a_ib_i)^p \right] = \left[ \varphi(a_i), \varphi(a_jb_j)^q \right] \\ \\ & = & \displaystyle \varphi \left( \left[ a_i, (a_jb_j)^q \right] \right) = \varphi(1) = 1. \end{array}$$
Thus, $\varphi$ is not injective on $D_i$, which implies that $\varphi(D_i)$ must be finite, a contradiction. 
\end{proof}

\begin{cor}\label{cor:AffineCoxeter}
For every $n \geq 2$, the affine Coxeter group $\widetilde{A_n}$ satisfies $(\mathrm{FW}_{n})$. 
\end{cor}

\noindent
During the proof of the corollary, we will use the following description of $\widetilde{A_n}$. First, set $G_n:= \mathbb{Z}^{n+1} \rtimes S_{n+1}$, where the symmetric group $S_{n+1}$ permutes the generators of a free basis of $\mathbb{Z}^{n+1}$. The subgroup $\Lambda_0 \leq \mathbb{Z}^{n+1}$ corresponding to the $(n+1)$-tuples whose sum equals $0$ yields a subgroup of rank $n$ normalised by $S_{n+1}$ in $G_n$. Let $C_n$ denote the subgroup $\Lambda_0 \rtimes S_{n+1}$ of $G_n$. Then $C_n$ is isomorphic to the affine Coxeter group $\widetilde{A_n}$. 

\begin{proof}[Proof of Corollary~\ref{cor:AffineCoxeter}.]
According to Theorem~\ref{thm:VirtuallyAbelian}, it suffices to show that no subgroup $H \leq C_n$ of index $\leq n$ has a morphism to $\mathbb{D}_\infty$ with infinite image. So fix a subgroup $H \leq C_n$ of index $\leq n$ and a morphism $\varphi : H \to \mathbb{D}_\infty$. We want to prove that the image of $\varphi$ is necessarily finite. For $n \neq 3$, the assertion follows from the next observation:

\begin{claim}\label{lem:Symmetric}
Assume that $n\neq 3$ and let $L$ be a group with no subgroup isomorphic to the alternating group $A_{n+1}$. Every morphism $\varphi : H \to L$ has finite image.
\end{claim}

\noindent
First, notice that $A_{n+1}$ has no proper subgroup of index $\leq n$. Otherwise, we would get a non-trivial morphism $A_{n+1} \to S_m$ for some $m\leq n$. Notice that 
$$(n+1)! = (n+1) \cdot n! > 2 \cdot m!, \text{ hence } |A_n| > |S_m|.$$
Consequently, our morphism $A_{n+1} \to S_m$ cannot be injective, which implies, since $A_{n+1}$ is simple, that the morphism must be trivial. We deduce from this observation that $H$ must contain $A_{n+1} \leq S_{n+1}$. Therefore, $H$ must contain $(\Lambda_0 \cap (k\mathbb{Z})^{n+1}) \rtimes A_{n+1}$ as a finite-index subgroup for some large $k \geq 1$. In order to conclude the proof of our claim, it suffices to show that $\varphi$ is trivial on this subgroup. We already know that $\varphi$ cannot be injective on $A_{n+1}$, by assumption, so $\varphi$ must be trivial on $A_{n+1}$ by simplicity. Consider the elements
$$a:= (k,-k,0,0, \ldots, 0), \ b:=(0,k,-k,0, \ldots, 0), \text{ and } c:=(k,0,-k,0, \ldots, 0)$$
of $\Lambda_0 \cap (k\mathbb{Z})^{n+1}$. Because they lie in the same $A_{n+1}$-orbit, they must have the same image under $\varphi$. But $c=ab$, so 
$$\varphi(c)= \varphi(ab) = \varphi(a)\varphi(b)= \varphi(c)^2, \text{ hence } \varphi(c)=1.$$
But the $A_{n+1}$-orbit of $c$ generate $\Lambda_0 \cap (k\mathbb{Z})^{n+1}$, so the desired conclusion follows, proving Claim~\ref{lem:Symmetric}.

\medskip \noindent
For $n = 3$, the alternating group is no longer simple and we treat the case separately. Because $H \cap S_4$ has index at most $3$ in $S_4$, we know that $H \cap S_4$ is either $S_4$, or $A_4$, or a $2$-Sylow subgroup. In the latter case, because $2$-Sylow subgroups are always pairwise conjugate, we can assume without loss of generality that $H \cap S_4$ coincides with our favourite subgroup of order $8$, say the copy $\langle (1234), (12) \rangle$ of the dihedral group $D_4$. Thus, $H$ contains a finite-index subgroup of the form $(\Lambda_0 \cap (k\mathbb{Z})^4 ) \rtimes S_4$, $(\Lambda_0 \cap (k\mathbb{Z})^4 ) \rtimes A_4$, or $(\Lambda_0 \cap (k\mathbb{Z})^4 ) \rtimes D_4$. In order to conclude our proof, it suffices to show that these three groups, which are respectively isomorphic to $\Lambda_0 \rtimes S_4$, $\Lambda_0 \rtimes A_4$, and $\Lambda_0 \rtimes D_4$, do not admit a morphism to $\mathbb{D}_\infty$ having an infinite image. The first case follows from the second case. In order the prove the second case, notice that a morphism $\varphi : \Lambda_0 \rtimes A_4 \to \mathbb{D}_\infty$ must send $A_4$ to a cyclic subgroup of order $\leq 2$. But $A_4$ does not contain a subgroup of index two, so actually $\varphi$ must be trivial on $A_4$. Then, we can conclude by repeating word for word the end of the proof of Claim~\ref{lem:Symmetric}. It remains to consider $\Lambda_0 \rtimes D_4$. 

\medskip \noindent
An element of $\Lambda_0$ can be thought of as a labelling of the vertices of a $4$-cycle with integers that sum to zero. Then, the action of $D_4$ on $\Lambda_0$ coincides with the permutation of the integers according to the usual action of $D_4$ on the $4$-cycle. We fix generators $a,b,c$ of $\Lambda_0$ and $x,y,z$ of $D_4$ as follows:

\begin{center}
\includegraphics[width=0.7\linewidth]{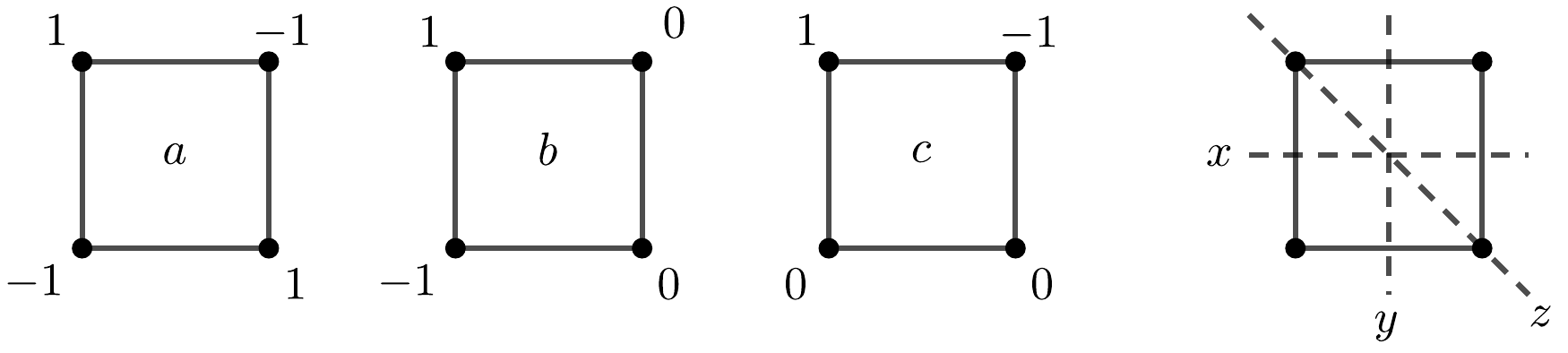}
\end{center}

\noindent
Then $\Lambda_0$ admits $\langle a,b,c \mid [a,b]=[b,c]=[a,c]=1 \rangle$ as a presentation, and $D_4$ admits 
$$\langle x,y,z \mid x^2=y^2=z^2=[x,y]=1, zxz^{-1}=y, zyz^{-1}=x \rangle$$
as a presentation. Consequently, $\Lambda_0 \rtimes D_4$ can be described as generated by $a,b,c,x,y,z$ submitted to the previous relations as well as
$$\left\{ \begin{array}{l} xax^{-1}=a^{-1} \\ yay^{-1}=a^{-1} \\ zaz^{-1}=a \end{array} \right., \ \left\{ \begin{array}{l} xbx^{-1}=b^{-1} \\ yby^{-1}=a^{-1}b \\ zbz^{-1}=c \end{array} \right., \text{ and } \left\{ \begin{array}{l} xcx^{-1}=a^{-1}c \\ ycy^{-1}=c^{-1} \\ zcz^{-1}=b \end{array} \right..$$
Notice that a morphism $\varphi : \Lambda_0 \rtimes D_4 \to \mathbb{D}_\infty$ must send $D_4$ to a cyclic subgroup of order $\leq 2$, so the kernel of $\varphi$ must contain a subgroup of index two in $D_4$. But $D_4$ has only three subgroups of index two, namely, when thinking of $D_4$ as a subgroup of $S_4$ (the latter permuting the vertices of our $4$-cycle, labelled cyclically from $1$ (the top left vertex) to $4$ (the bottom left vertex)):
$$\left\{ \mathrm{id}, \ x=(12)(34), \ y=(14)(23), \ (13)(24) \right\},$$
$$\left\{ \mathrm{id}, \ (13), \ z= (24),\ (13)(24) \right\},$$
and
$$\left\{ \mathrm{id}, \ xzy= (1234), \ (13)(24), \ (1432) \right\}.$$ 
Therefore, in order to conclude that the image of $\varphi$ is finite, it suffices to show that putting $x=y=1$, or $z=1$, or $x=yz$ in the presentation of $\Lambda_0 \rtimes D_4$ given above yields a finite group. 

\medskip \noindent
First, let us add the relations $x=y=1$. Then we deduce from $a=xax^{-1}=a^{-1}$ that $a$ becomes an element of order two. Same thing for $b$ and $c$. Thus, the group we get has a normal abelian $2$-subgroup $\langle a,b,c \rangle$ with a corresponding quotient that is finite (as a quotient of $D_4$), hence the desired conclusion.

\medskip \noindent
Next, let us add the relation $z=1$. We have $b= zbz^{-1}=c$. Therefore, 
$$c^{-1}= b^{-1}= xbx^{-1}=xcx^{-1}=a^{-1}c, \text{ hence } a=c^2,$$
which implies that
$$b^{-1}= xax^{-1}= xb^2x^{-1}=(xbx^{-1})^2 = b^{-2}, \text{ hence } b=1.$$
Consequently, $a=b=c=1$. We conclude that the group we obtain from our new presentation is a quotient of $D_4$, and therefore finite, as desired.

\medskip \noindent
Finally, let us add the relation $x=yz$. As a consequence, 
$$b^{-1}=xbx^{-1}= y (zbz^{-1}) y^{-1} = ycy^{-1}= c^{-1}, \text{ hence } b=c.$$
Also, from
$$1= [x,y]= [yz,y]= y [z,y] y^{-1}, \text{ hence } [y,z]=1;$$
we deduce that $x=zyz^{-1}=y$. This implies
$$b^{-1}= xbx^{-1}=yby^{-1}=a^{-1}b, \text{ hence } a=b^2.$$
So far, we have proved that $b=c$ and $a=c^2$. But we saw in the previous case that these relations imply that our group is finite, which concludes the proof. 
\end{proof}

\section{Cubulable groups with $(\mathrm{FW}_n)$}

\noindent
In this section, our goal is to exhibit a source of groups that are cubulable in some dimension but that also satisfy a fixed-point property in other dimensions. In Subsection~\ref{section:VGP}, we describe these groups and explain why they are cubulable. In Subsection~\ref{section:VGPfix}, we show that some of these groups satisfy good fixed-point properties.

\subsection{Virtual graph products}\label{section:VGP}

\noindent
Recall that, given a graph $\Gamma$ and a collection of groups $\mathcal{G}= \{G_u \mid u \in V(\Gamma)\}$ indexed by the vertices of $\Gamma$, the \emph{graph product} $\Gamma \mathcal{G}$ is defined as
$$\left( \underset{u \in V(\Gamma)}{\ast} G_u \right) / \langle \langle [g,h], \ g \in G_u, h \in G_v, \{u,v\} \in E(\Gamma) \rangle \rangle.$$
When all the groups in $\mathcal{G}$ are identical, say to some group $G$, we write $\Gamma G$ instead of $\Gamma \mathcal{G}$ for simplicity.

\medskip \noindent
Usually, one says that graph products interpolate between free products (when $\Gamma$ has no edge) and direct sums (when $\Gamma$ is a complete graph). Right-angled Artin and Coxeter groups are exemples of graph products. Most graph products split non-trivially as amalgamated products, and consequently act non-trivially on trees, which imply that they have rather bad median fixed-point properties. Nevertheless, it turns out that good fixed-point properties may be obtained when twisting graph products just a little bit.

\medskip \noindent
Given a graph $\Gamma$, a group $G$, and a subgroup $H \leq \mathrm{Isom}(\Gamma)$, define the \emph{virtual graph product}
$$\Gamma[G,H]:= \Gamma G \rtimes H$$
where $H$ acts on $\Gamma G$ by permuting the factors $G$ according to its action on $\Gamma$. If $H= \mathrm{Isom}(\Gamma)$, we denote our group simply by $\Gamma[G]$. 

\medskip \noindent
Fixed-point properties of virtual graph products are studied in the next subsection. In the opposite direction, we end this subsection by proving that virtualy graph products of cubulable groups are cubulable.

\begin{thm}\label{thm:VGPcc}
Let $\Gamma$ be a finite graph, $G$ a group, and $H \leq \mathrm{Isom}(\Gamma)$ a subgroup. If $G$ acts properly and cocompactly on a median graph of cubical dimension $d$, then $\Gamma[G,H]$ admits a proper and cocompact action on a median graph of cubical dimension $d\cdot \mathrm{clique}(\Gamma)$.
\end{thm}

\noindent
Before considering Theorem~\ref{thm:VGPcc} in full generality, let us assume first that $G$ is a finite group. In this case, it is known that the graph product $\Gamma G$ acts properly and cocompactly on a median graph. Namely, let $M$ be the graph
\begin{itemize}
	\item whose vertices are the cosets $g\langle \Lambda \rangle$ where $g \in \Gamma G$ and $\Lambda \subset \Gamma$ complete;
	\item and whose edges connect two cosets $g \langle \Lambda \rangle$ and $g \langle \Lambda \cup \{ \mathrm{pt}\} \rangle$.
\end{itemize}
Then $M$ is a median graph and $\Gamma G$ acts properly and cocompactly on $M$ by left-multiplication. See for instance \cite{MR2240922} and references therein for more details. Because $\mathrm{Isom}(\Gamma)$, when thought of as a subgroup of $\mathrm{Aut}(\Gamma G)$, permutes the subgroups in $\{ \langle \Lambda \rangle, \ \Lambda \subset \Gamma \text{ complete}\}$, the action of $\Gamma G$ on $M$ naturally extends to $\Gamma[G]$, which is again proper and cocompact. 

\medskip \noindent
When $G$ is an infinite group, the construction of $M$ still makes sense and still produces a median graph, but the action of $\Gamma G$ is no longer proper. So a new construction a necessary. This is what we explain below. Nevertheless, the concrete examples mentioned in the next section essentially concern only virtual graph products of finite groups, so the general proof below can be skipped during a first reading.

\medskip \noindent
The general case of Theorem~\ref{thm:VGPcc} can be proved by following the lines of \cite[Theorem~8.17]{QM}, which proves that graph products of cubulable groups are cubulable. The argument exploits the \emph{quasi-median} geometry of graph products. In the same way that median graphs can be defined as retracts of infinite cubes, quasi-median graphs can be defined as retracts of infinite prisms (i.e.\ products of (infinitely many infinite) complete graphs). As shown in \cite{QM}, most of the machinery of hyperplanes in median graphs naturally generalises to quasi-median graphs. We record a few definitions and properties needed in the proof of Theorem~\ref{thm:VGPcc} given below.
\begin{itemize}
	\item In a quasi-median graph, a \emph{hyperplane} is an equivalence class of edges with respect to the reflexive-transitive closure of the relation that identifies two edges whenever they belong to a common $3$-cycle or whenever they are opposite edges in a $4$-cycle. It follows from the definition that every complete subgraph is entirely contained in a single hyperplane. In fact, one can think of a hyperplane as a collection of parallel \emph{cliques} (i.e.\ maximal complete subgraphs). 
	\item Given a quasi-median graph $X$ and a hyperplane $J$, the graph $X \backslash \backslash J$ obtained by removing the edges from $J$ is disconnected. More precisely, given a clique $C$ whose edges belong to $J$, the vertices of $C$ lie in pairwise distinct components of $X \backslash \backslash J$, called \emph{sectors}; and, conversely, every sector contains a vertex of $C$. 
	\item Cliques in quasi-median graphs are \emph{gated}. (Recall that, in a graph $X$, a subgraph $Y \subset X$ is \emph{gated} if, for every $x \in X$, there exists some $y \in Y$ such that, for every $z \in Y$, some geodesic connecting $x$ to $z$ passes through $y$. The vertex $y$ is referred to as the \emph{projection of $x$ to $Y$}.)
	\item In a quasi-median graph, a \emph{prism} refers to a product of cliques. 
\end{itemize}
We refer the reader to \cite{QM} for more details. 

\begin{proof}[Proof of Theorem~\ref{thm:VGPcc}.]
By assumption, $G$ admits a proper and cocompact action on some median graph $X$. Up to adding \emph{spikes} to $X$, we can assume without loss of generality that there exists some vertex in $X$ whose $G$-stabiliser is trivial. (See for instance \cite[Lemma~4.34]{QM}.) Pull back through the orbit map $g \mapsto g x_0$ the wallspace structure of $X$ induced by its hyperplanes to a $G$-invariant wallspace structure $\mathcal{W}$ on $G$. The fact that $x_0$ has a trivial stabiliser implies that any two points of $G$ are separated by at least one wall of $\mathcal{W}$. See \cite[Section~4.3]{QM} for details. 

\medskip \noindent
According to \cite[Proposition~8.2]{QM}, the Cayley graph
$$\mathrm{QM}:= \mathrm{Cayl} \left( \Gamma G, \bigcup\limits_{u \in V(\Gamma)} G_u \right),$$
where $G_u$ denotes the copy of $G$ indexed by $u$, turns out to be quasi-median. The graph product $\Gamma G$ acts on $\mathrm{QM}$ by left-multiplication. Because $\mathrm{Isom}(\Gamma)$ permutes the generators, this action extends to $\Gamma[G]$. A vertex in $\mathrm{QM}$ has trivial stabiliser in $\Gamma G$, but its stabiliser in $\Gamma[G]$ is conjugate to $\mathrm{Isom}(\Gamma)$. In particular, $\Gamma[G]$ acts on $\mathrm{QM}$ with finite stabilisers. Let us record a few basic properties about $\mathrm{QM}$.
\begin{itemize}
	\item The cliques of $\mathrm{QM}$ are the cosets $gG_u$ where $g \in \Gamma G$, $u \in V(\Gamma)$. (See \cite[Lemma~8.6]{QM}.)
	\item The prisms of $\mathrm{QM}$ are the cosets $g\langle \Lambda \rangle$ where $g \in \Gamma G$, $\Lambda \subset \Gamma$ complete. (See \cite[Corollary~8.7]{QM}.)
	\item Two cliques $C_1$ and $C_2$ belong to the same hyperplane if and only if there exist $g \in \Gamma G$, $u \in V(\Gamma)$, and $h \in \langle \mathrm{link}(u) \rangle$ such that $C_1=gG_u$ and $C_2=gh G_u$. (See \cite[Lemma~8.9 and Corollary~8.10]{QM}.)
\end{itemize}
Now, we want to endow $\mathrm{QM}$ with a \emph{system of wallspaces}, i.e.\ we want to endow every clique of $\mathrm{QM}$ with a wallspace structure. For every clique $C$, there exists a unique vertex $u \in V(\Gamma)$ and a unique shortest element $g \in \Gamma G$ (with respect to the generating set $\bigcup_{v \in V(\Gamma)} G_v$) such that $C=gG_u$. This allows us to identify canonically $C$ with $G$ and to transfer the walls $\mathcal{W}$ on $G$ to walls $\mathcal{W}(C)$ on $C$. 

\medskip \noindent
Our system of wallspaces is \emph{coherent}. Namely, let $C_1$ and $C_2$ be two cliques that belong to the same hyperplane. We claim that the projection map $C_1 \to C_2$ is a bijection that sends walls to walls. We already know that there exist $g \in \Gamma G$, $u \in V(\Gamma)$, and $h \in \langle \mathrm{link}(u) \rangle$ such that $C_1=gG_u$ and $C_2=gh G_u$. We can choose $g$ of minimal length (which amounts to saying that, when $g$ is written as a graphically reduced word, no syllable belonging to $G_u$ can be shifted to the end of the word). If so, $gh$ also has minimal length. Then the projection $C_1 \to C_2$ can be simply written as $gs \mapsto ghs$ for every $s \in G_u$. Walls are sent to walls essentially by construction.

\medskip \noindent
Next, notice that our system of wallspaces is $\Gamma[G]$-invariant. Indeed, let $C$ be a clique and let $W \in \mathcal{W}(C)$ be a wall. Given an element $(g,\varphi) \in \Gamma[G]$, we claim that $(g, \varphi) \cdot W$ is a wall in $\mathcal{W}((g,\varphi) C)$. Fix a vertex $u \in V(\Gamma)$ and an element $h \in \Gamma G$ of minimal length such that $C=hG_u$. We have $(g,\varphi) \cdot C = gh^\varphi G_{\varphi(u)}$. The element $gh^\varphi$ may no longer have minimal length, but we can decompose it as a product $ks$ where $s \in G_{\varphi(u)}$ and where $k$ now has minimal length. By construction, our wall $W$ can be written as $h \{ W^-,W^+\}$ for some $\{W^-,W^+\}$ in $\mathcal{W}$ (when $G_u$ is identified with $G$). Then $(g,\varphi) \cdot W = k \{sW^-,sW^+\}$. But $\mathcal{W}$ is $G$-invariant, so $\{sW^-,sW^+\}$ belongs to $\mathcal{W}$, which implies that $(g,\varphi) \cdot W$ belongs to $\mathcal{W}(C)$, as desired.

\medskip \noindent
As explained in \cite[Section~4.1]{QM}, the system of wallspaces $\{ (C, \mathcal{W}(C)) \mid C \text{ clique}\}$ can be extended to a wallspace structure $\mathcal{HW}$ on $\mathrm{QM}$. The idea is the following. For every clique $C$, set
$$\overline{\mathcal{W}}(C):= \left\{ \left\{ \mathrm{proj}_C^{-1}(W^-), \mathrm{proj}_C^{-1}(W^+) \right\}, \ \{W^-,W^+\} \in \mathcal{W}(C) \right\}.$$
Because our system is coherent, the collection of walls $\overline{\mathcal{W}}(C)$ depends only on the hyperplane $J$ containing $C$, so we can write it $\mathcal{W}(J)$. Our walls on $\mathrm{QM}$ are now $\mathcal{HW}:= \{ \mathcal{W}(J), \text{ $J$ hyperplane}\}$. Our group $\Gamma[G]$ acts on the wallspace $(\mathrm{QM}, \mathcal{HW})$ by preserving walls. 

\medskip \noindent
Let $M$ denote the median graph obtained by cubulating $(\mathrm{QM},\mathcal{HW})$. It follows from \cite[Propositions~4.22 and~4.23]{QM} that the action of $\Gamma[G]$ on $M$, induced by the action of $\Gamma[G]$ on $\mathrm{QM}$, is proper and cocompact. Moreover, notice that \cite[Corollary~4.14, Lemma~3.1, and Corollary~8.7]{QM} imply that the cubical dimension of $M$ is at most $d \cdot \mathrm{clique}(\Gamma)$. 

\medskip \noindent
Thus, we have constructed the desired cubulation for $\Gamma[G]$. The same conclusion holds for $\Gamma[G,H]$, given a subgroup $H \leq \mathrm{Isom}(\Gamma)$, since $\Gamma[G,H]$ has finite index in $\Gamma[G]$. 
\end{proof}

\subsection{Fixed-point property}\label{section:VGPfix}

\noindent
In this subsection, our goal is to exhibit examples of virtual graph products that satisfy good median fixed-point properties. Our main results are Theorems~\ref{thm:GraphProdFW} and~\ref{thm:GraphProdFWstraight}. Before statement our first theorem, we need to introduce some terminology.

\begin{definition}
A group $G$ satisfies the property $(\mathrm{FW}_n^+)$ if, for every action of $G$ on a median graph without hyperplane-inversions, $\mathrm{Fix}(G)$ is non-empty and convex.
\end{definition}

\noindent
This property will be useful for us because the Helly property for convex subgraphs in median graphs will allow us to prove that some intersections of fixators are non-empty, producing global fixed points. 

\begin{lemma}\label{lem:FixConvex}
A group $G$ satisfies $(\mathrm{FW}_n^+)$ if and only if it satisfies $(\mathrm{FW}_n)$ and it does not contain proper subgroups of index $\leq n$.
\end{lemma}

\begin{proof}
If $G$ satisfies $(\mathrm{FW}_n^+)$, then it clearly satisfies $(\mathrm{FW}_n)$. Let $H \leq G$ be a proper finite-index subgroup. Let $G$ act on the cube $Q:=[0,1]^{G/H}$ by permuting the coordinates. Notice that the vertices $\bar{0}:=(0,\ldots, 0)$ and $\bar{1}:=(1,\ldots, 1)$ are fixed by $G$, that the convex hull of $\{\bar{0},\bar{1}\}$ is $Q$ entirely, and that $G$ does not act trivially on $Q$ since $H$ is a proper subgroup. Therefore, it follows from the property $(\mathrm{FW}_n^+)$ that $[G:H] = \dim(Q) \geq n$. 

\medskip \noindent
Conversely, assume that $G$ satisfies $(\mathrm{FW}_n)$ and that it does not contain proper subgroups of index $\leq n$. We want to prove that $G$ satisfies $(\mathrm{FW}_n^+)$, so we fix an action $G \curvearrowright X$ on a median graph of cubical dimension $\leq n$ with no hyperplane-inversion. Since $G$ has $(\mathrm{FW}_n)$, we already know that $\mathrm{Fix}(G)$ is non-empty, and we claim that $\mathrm{Fix}(G)$ is furthermore convex. According to \cite{MR1114014}, it suffices to verify that $\mathrm{Fix}(G)$ is connected and \emph{locally convex} (i.e.\ every $4$-cycle containing two successive edges in $\mathrm{Fix}(G)$ is entirely contained in $\mathrm{Fix}(G)$). 

\medskip \noindent
We begin by proving that $\mathrm{Fix}(G)$ is connected. So let $u,v \in \mathrm{Fix}(G)$ be two vertices. We want to connect $u$ and $v$ by a path in $\mathrm{Fix}(G)$. If $u=v$, then there is nothing to prove, so from now on we assume that $u \neq v$. Let $a_1, \ldots, a_k \in X$ denote all the neighbours of $u$ that belong to geodesics connecting $u$ to $v$. Because the hyperplanes dual to the edges $[u,a_i]$, $1 \leq i \leq k$, must be pairwise transverse, the convex hull of $\{u,a_1, \ldots, a_k\}$ must be a $k$-cube, say $Q$. Notice that the vertex $u'$ opposite to $u$ in $Q$ belongs to a geodesic from $u$ to $v$, hence $d(v,u')<d(v,u)$. Proving that $Q \subset \mathrm{Fix}(g)$ will allow us to conclude by induction over $d(v,u)$ that there exists a path in $\mathrm{Fix}(G)$ connecting $u$ to $v$. Since the construction of $Q$ depends only on $u$ and $v$, which are fixed by $G$, necessarily $Q$ is stabilised by $G$. But $G$ also fixes $u \in Q$, so the action of $G$ on $Q$ is uniquely determined by the permutations of the $a_i$ it induces. The fixator of each $a_i$ yields a subgroup of index $\leq k \leq n$ in $G$. As a consequence of our assumptions, such a subgroup cannot be proper, so $G$ must fix each $a_i$, hence $Q \subset \mathrm{Fix}(G)$. We conclude that $\mathrm{Fix}(g)$ is connected, as desired.

\medskip \noindent
The local convexity of $\mathrm{Fix}(G)$ is clear. Indeed, if $a \in \mathrm{Fix}(G)$ is a vertex and $b,c \in \mathrm{Fix}(G)$ are two neighbours such that the edges $[a,b]$ and $[a,c]$ span a $4$-cycle $C$, then $C$ is the unique $4$-cycle spanned by $[a,b]$ and $[a,c]$, so the fourth vertex of $C$ is necessarily fixed by $G$. 
\end{proof}

\noindent
As a straight consequence of Lemma~\ref{lem:FixConvex}, we get the following examples:

\begin{cor}\label{cor:FWPlus}
For all $n \geq 1$ and $q \geq 2$, if $q$ has no divisor in $[2,n]$ then the cyclic group $\mathbb{Z}/q\mathbb{Z}$ satisfies $(\mathrm{FW}_n^+)$. 
\end{cor}

\noindent
We are now ready to state and prove our first theorem.

\begin{thm}\label{thm:GraphProdFW}
Let $n \geq 1$ be an integers, $G$ a group, $\Gamma$ a finite graph, and $H \leq \mathrm{Isom}(\Gamma)$ a subgroup. Assume that the following assertions hold:
\begin{itemize}
	\item $G$ satisfies $(\mathrm{FW}_n^+)$ and $\Gamma$ has diameter two;
	\item the flag-completion of $\Gamma$ has some non-trivial topology in dimension $n$;
	\item for all vertices $a,b,x,y \in \Gamma$ satisfying $d(a,b)=d(x,y) =2$, there exists some $h \in H$ such that $h \cdot \{a,b\}= \{x,y\}$.
\end{itemize}
Then $\Gamma[G,H]$ satisfies the property $(\mathrm{FW}_n)$. 
\end{thm}

\noindent
A CW-complex $X$ has \emph{some non-trivial topology in dimension $n$} if there exists a property $\mathcal{P}$ such that:
\begin{itemize}
	\item $\mathcal{P}$ is invariant under homotopy;
	\item $\mathcal{P}$ is satisfied by $X$ but not by contractible spaces;
	\item every complex obtained from $X$ by adding (possibly infinitely many) cells in dimensions $\leq n$ still satisfies $\mathcal{P}$.
\end{itemize}
For instance, a complex whose $n$th homotopy or (co)homology group is non-trivial has some non-trivial topology in dimension $n$. 

\begin{proof}[Proof of Theorem~\ref{thm:GraphProdFW}.]
Let $X$ be a median graph of cubical dimension $n$ on which $\Gamma[G,H]$ acts. Up to replacing $X$ with a subdivision, we assume that there is no hyperplane-inversion. For every vertex-group $G_v$ ($v \in V(\Gamma)$) of $\Gamma G$, we know by assumption that $\mathrm{Fix}(G_v)$ is convex in $X$. Moreover, given two adjacent $u,v \in V(\Gamma)$, we know that $\langle G_u,G_v \rangle \simeq G_u \times G_v$ has to fix a vertex, which amounts to saying that $\mathrm{Fix}(G_u) \cap \mathrm{Fix}(G_v) \neq \emptyset$. Now, we need to distinguish two cases.

\medskip \noindent
First, assume that there exists two non-adjacent $u,v \in V(\Gamma)$ such that $\mathrm{Fix}(G_u) \cap \mathrm{Fix}(G_v) \neq \emptyset$. Because $\Gamma$ has diameter two and because of our transitivity assumption on $\Gamma$, it follows that $\mathrm{Fix}(G_u) \cap \mathrm{Fix}(G_v) \neq \emptyset$ for all vertices $u,v \in V(\Gamma)$. It follows from the Helly property for convex subgraphs that $\bigcap_{v \in V(\Gamma)} \mathrm{Fix}(G_v)$ is non-empty, which amounts to saying that there exists a vertex in $X$ fixed by $\Gamma G$ entirely. We conclude that $\Gamma[G:H]$ has a global fixed point in $X$.

\medskip \noindent
Next, assume that $\mathrm{Fix}(G_u) \cap \mathrm{Fix}(G_v) \neq \emptyset$ if and only if $u$ and $v$ are adjacent in $\Gamma$. Let $Y$ denote the union $\bigcup_{u \in V(\Gamma)} \mathrm{Fix}(u)$ and let $Y^\square,X^\square$ denote respectively the cube-completions of $Y,X$. Our assumption precisely means that the nerve complex of the cube-completions of the $\mathrm{Fix}(G_u)$, $u \in V(\Gamma)$, in $Y^\square$ is isomorphic to the flag completion $\Gamma^\triangle$ of $\Gamma$. Because a non-empty intersection of $\mathrm{Fix}(G_u)$ is convex, the corresponding cube-completion must be contractible. It follows from Leray's nerve theorem that $Y^\square$ is homotopy equivalent to $\Gamma^\triangle$. Thus, $Y^\square$ has some non-trivial topology in dimension $n$. But $X^\square$ is obtained from $Y^\square$ by adding cells in dimensions $\leq n$, so $X^\square$ cannot be contractible, a contradiction. 
\end{proof}

\noindent
As an application, we can construct examples of groups that are cubulable in dimension $n$ but with the fixed-point property $(\mathrm{FW}_{n-1})$ in dimension one less.

\begin{cor}\label{cor:CubulableButFW}
Let $n \geq 2$ be an integer and let $\Gamma$ denote the join of $n$ copies of the graph having only two isolated vertices. For every $q \geq 2$ with no divisor in $[2,n]$, the group
$$\Gamma(\mathbb{Z} /q \mathbb{Z}) \rtimes \mathrm{Isom}(\Gamma)$$
admits a proper and cocompact action on a median graph of cubical dimension $n$ but satisfies $(\mathrm{FW}_{n-1})$. 
\end{cor}

\begin{proof}
Notice that $\Gamma$ is the one-skeleton of a triangulation of the sphere $\mathbb{S}^{n-1}$. Thus, using also Corollary~\ref{cor:FWPlus}, we verify easily that Theorem~\ref{thm:GraphProdFW} applies, proving that our group satisfies $(\mathrm{FW}_{n-1})$. The fact that it is cubulable in dimension $n$ follows from Theorem~\ref{thm:VGPcc}. 
\end{proof}

\noindent
Theorem~\ref{thm:GraphProdFW} provides a lot of possible examples, but the restriction for the graph $\Gamma$ to have diameter two imposes severe restrictions on the geometry of the resulting virtual graph product. For instance, the examples given by Corollary~\ref{cor:CubulableButFW} are all products of virtually free groups. Our next theorem shows that the strategy used in order to prove Theorem~\ref{thm:GraphProdFW} can be applied to some graphs with large diameters. For instance, it produces acylindrically hyperbolic examples.

\begin{thm}\label{thm:GraphProdFWstraight}
Let $r,s,n \geq 1$ satisfy $r\geq s$ and $s \geq n$, and let $\Gamma:=\Gamma_{r,s}$ denote the graph obtained from an $r$-cube by connecting with an edge any two vertices at distance $\leq s$. For every group $G$ satisfying $(\mathrm{FW}_n^+)$, the virtual graph product $\Gamma[G]$ satisfies $(\mathrm{FW}_n)$ but, if $G$ acts properly and cocompactly on some median graph, then so does $\Gamma[G]$. 
\end{thm}

\begin{proof}
Let $X$ be a median graph of cubical dimension $n$ on which $\Gamma[G]$ acts. Up to replacing $X$ with its cubical subdivision, we can assume that every elliptic subgroup of $\Gamma[G]$ fixes a vertex. For every vertex-group $G_v$ ($v \in V(\Gamma)$) of $\Gamma G$, we know by assumption that $\mathrm{Fix}(G_v)$ is convex in $X$. Moreover, given two adjacent $u,v \in V(\Gamma)$, we know that $\langle G_u,G_v \rangle \simeq G_u \times G_v$ has to fix a vertex, which amounts to saying that $\mathrm{Fix}(G_u) \cap \mathrm{Fix}(G_v) \neq \emptyset$.

\medskip \noindent
Let $\Gamma_0$ denote the $r$-cube from which $\Gamma$ is constructed. If the $\mathrm{Fix}(G_u)$, $u \in V(\Gamma)$, pairwise intersect, then the Helly property for convex subgraphs imply that there exists a vertex fixed by all the factors $G_u$, $u \in V(\Gamma)$, of $\Gamma G$, which implies that $\Gamma[G]$ has a global fixed point. So from now on, we assume that there exist at least two vertices $u,v \in V(\Gamma)$ such that $\mathrm{Fix}(G_u) \cap \mathrm{Fix}(G_v) = \emptyset$. We choose the vertices $u$ and $v$ with minimal distance, say $d$, in $\Gamma_0$; notice that this distance must be $>s$ because $u$ and $v$ cannot be adjacent in $\Gamma$. Let $Q_0 \subset \Gamma_0$ denote the smallest subcube containing $u$ and $v$, and let $Q \subset \Gamma$ denote the corresponding induced subgraph. 

\medskip \noindent
Let $Y \subset X$ denote the union of the $\mathrm{Fix}(G_a)$ for $a \in Q$. By minimality of the distance between $u$ and $v$, and because $\mathrm{Isom}(\Gamma)$ permutes transitively the vertices of $\Gamma_0$ at distance $d$, the one-skeleton of the nerve complex coincides with the graph obtained from $Q_0$ by connecting with an edge any two vertices that are not opposite in the cube $Q_0$. Consequently, this graph is the join of $d$ pairs of isolated vertices (corresponding to the pairs of opposite vertices), which implies that our nerve complex is a $(d-1)$-sphere. Because a non-empty intersection of $\mathrm{Fix}(G_a)$ is convex, the corresponding cube-completion must be contractible. It follows from Leray's nerve theorem that the cube-completion $Y^\square$ of $Y$ is homotopy equivalent to a $(d-1)$-sphere. 

\medskip \noindent
But the cube-completion $X^\square$ of $X$ can be obtained from $Y^\square$ by adding cells of dimensions $\leq n$. As $n \leq s <d$, this operation cannot kill our $(d-1)$-sphere homotopically, contradicting the fact that $X^\square$ is contractible.

\medskip \noindent
Thus, we have proved that $\Gamma[G]$ satisfies $(\mathrm{FW}_n)$. The last assertion of our theorem is an immediate consequence of Theorem~\ref{thm:VGPcc}. 
\end{proof}

\section{Final comments and questions}\label{section:Questions}

\noindent
In this final section, we discuss the results proved in the article and record some open questions and onjectures related to median fixed-point properties.

\paragraph{Other sources of examples.} As mentioned in the introduction, median fixed-point properties are fully understood only for a few families of groups, and it would be interesting to find other sources of examples. For instance, \cite{CornulierCommensurated} shows that $\mathrm{SL}_2(\mathbb{Z}[\sqrt{2}])$ satisfies $(\mathrm{FW})$ (but not $(\mathrm{T})$) by noticing that it is boundedly generated by distorted abelian subgroups. Is it possible to find examples with $(\mathrm{FW})$ but not $(\mathrm{T})$ not based on distortion? For instance:

\begin{question}
Does there exist a finitely generated group with $(\mathrm{FW})$ but not $(\mathrm{T})$ all of whose finitely generated abelian subgroups are undistorted? 
\end{question}

\begin{question}
Does there exist a hyperbolic or CAT(0) group with $(\mathrm{FW})$ but not $(\mathrm{T})$?
\end{question}

\noindent
The examples constructed in this article are generated by finite-order elements, and this is fundamental in our arguments. Therefore, in order to find groups of different nature, it is natural to ask for torsion-free examples.

\begin{problem}\label{prob:TorsionFree}
For every $n \geq 1$, construct a torsion-free group with $(\mathrm{FW}_n)$ but acting properly and cocompactly on a median graph of cubical dimension $n+1$. 
\end{problem}

\noindent
A plausible source of torsion-free finitely generated groups with $(\mathrm{FW}_n)$ but not $(\mathrm{FW}_{n+1})$ is given by Bieberbach groups (i.e.\ torsion-free virtually abelian groups). Investigating these groups would be interesting.

\medskip \noindent
Notice that, among the groups given by Theorems~\ref{thm:VirtuallyAbelian}, \ref{thm:GraphProdFW}, and~\ref{thm:GraphProdFWstraight}, none are hyperbolic. (Notice, however, that Theorem~\ref{thm:GraphProdFWstraight} provides acylindrically hyperbolic groups.) 

\begin{problem}
For every $n \geq 1$, construct a hyperbolic group with $(\mathrm{FW}_n)$ but acting properly and cocompactly on a median graph of cubical dimension $n+1$.
\end{problem}

\noindent
In this direction, an intriguing problem is the understanding of median fixed-point properties of (classical) small cancellation groups. We know from \cite{MR2053602} that they act properly and cocompactly on median graphs, but, in the construction, the cubical dimension increases with the lengths of the relations. Moreover, it follows from \cite{MR4108839} that the cubical dimension cannot be bounded uniformly and, according to \cite{MR0707619}, some small cancellation groups verify $(\mathrm{FA})$. Understanding when a small cancellation group satisfies (or not) the property $(\mathrm{FW}_n)$ is a good training problem.

\medskip \noindent
It might be frustrating that the examples provided by Theorems~\ref{thm:VirtuallyAbelian}, \ref{thm:GraphProdFW}, and~\ref{thm:GraphProdFWstraight} all virtually act non-trivially on trees. In other words, our examples crucially exploit the sensibility of $(\mathrm{FW}_n)$ to commensurability. This motivates the following question:

\begin{question}\label{question:FWandFI}
For every $n \geq 2$, does there exist a group acting properly and cocompactly on a median graph of cubical dimension $n$ but all of whose finite-index subgroup satisfy $(\mathrm{FW}_{n-1})$?
\end{question}

\noindent
It is worth mentioning that constructing a cubulable group all of whose finite-index subgroups satisfy $(\mathrm{FA})$ is already non-trivial. The main reason is that many of the cubulable groups we are familiar with are virtually special, and consequently virtually act on trees, and that most of the exotic cubulable groups we know are lattices in products of trees \cite{MR1446574, MR2694733,MR3931408}. Answering a question of I. Chatterji, the first example of a cubulable group all of whose finite-index have $(\mathrm{FA})$ has been recently constructed in \cite{MR4481088}. 

\begin{conj}\label{Conj}
For every $n \geq 1$, there exists an infinite cubulable group satisfying $(\mathrm{FW}_n)$ with no proper finite-index subgroup. 
\end{conj}

\noindent
Our conjecture provides a (weak) positive answer to Question~\ref{question:FWandFI}. We explain below why this statement is plausible.

\paragraph{A detour to median small cancellation.} A classical strategy used to produce examples satisfying various fixed-point properties is to create common small cancellation quotients. More precisely, let $A$ and $B$ be two groups. Given two sets of generators $a_1, \ldots, a_n \in A$ and $b_1, \ldots, b_m \in B$, the group
$$Q:= (A \ast B) / \langle \langle a_1=v_1, \ldots, a_n=v_n, b_1=u_1, \ldots, b_m=u_m \rangle \rangle,$$
where $u_1, \ldots, u_m \in A$ and $v_1, \ldots, v_n \in B$ are elements thought of as parameters of the construction, provides a common quotient of $A$ and $B$. When working in a family of groups with a good theory of small cancellation, we can choose the $u_i$ and $v_j$ ``sufficiently complicated'' in order to get some control on $Q$. Regarding Conjecture~\ref{Conj}, we would like to take $A$ with $(\mathrm{FW}_n)$ and $B$ with no proper finite-index subgroup. If we are able to show that $Q$ is infinite and cubulable, then for free it satisfies $(\mathrm{FW}_n)$ and it has no proper finite-index subgroup. Examples of cubulable groups without proper finite-index that admit small cancellation quotients are for instance provided by free products of Wise's or Burger-Mozes' lattices in products of trees \cite{MR1446574, MR2694733}; and examples of cubulable groups with $(\mathrm{FW}_n)$ that admit small cancellation quotients are given by Theorem~\ref{thm:GraphProdFWstraight}. So it only remains to understand the median geometry of small cancellation quotients of cubulable groups.

\medskip \noindent
In this perspective, Wise's cubical small cancellation \cite{MR4298722} provides a good framework. Notice, however, that the theory only applies to fundamental groups of non-positively curved cube complexes, and consequently only to torsion-free groups. Therefore, this does not apply to the groups given by Theorem~\ref{thm:GraphProdFWstraight}\footnote{But this could apply to groups solving Problem~\ref{prob:TorsionFree}.}. Nevertheless, in the spirit of the geometric small cancellation described in \cite{MR3589159}, it is possible to reformulate the cubical small cancellation in the context of groups acting on median graphs with contracting isometries. This is done in part in \cite{Book}, and we plan to pursue in this direction in future works. As another approach, more elementary, it is possible to extend Brady and McCammond's idea to cubulate $C'(1/4)-T(4)$ groups (unpublished). This can be applied to groups acting on median graphs, but, unfortunately, the condition $T(4)$ is usually not satisfied by our virtual graph products. It could, however, apply to other groups.

\begin{conj}\label{ConjSC}
Two acylindrically hyperbolic cubulable groups admit a common quotient that is acylindrically hyperbolic and cubulable.
\end{conj}

\noindent
Proving Conjecture~\ref{ConjSC} would imply Conjecture~\ref{Conj}. Notice that Conjecture~\ref{ConjSC} has other interesting consequences. For instance, it implies that acylindrically hyperbolic cubulable groups admit many cubulable quotients. More precisely, given an acylindrically hyperbolic cubulable group $G$, if Conjecture~\ref{ConjSC} is true, then:
\begin{itemize}
	\item every cubulable group embeds into a cubulable quotient of $G$;
	\item for every $n \geq 1$, $G$ admits an acylindrically hyperbolic cubulable group satisfying $(\mathrm{FW}_n)$;
	\item for every $n \geq 1$, there exists a proper cubulable quotient $G \twoheadrightarrow G_n$ that is injective on balls of radius $n$.
\end{itemize}
Conjecture~\ref{ConjSC} would also imply interesting embedding theorems. For instance, every cubulable group would embed into a $2$-generated cubulable group. In the same direction, we record the following interesting question:

\begin{question}\label{question:Simple}
Does every cubulable group embeds into a simple cubulable group?
\end{question}

\noindent
Proving Conjecture~\ref{ConjSC} would not be sufficient to answer positively this question, since simple groups are quite different from acylindrically hyperbolic groups. However, answering Question~\ref{question:Simple} would improve our understanding of simple cubulable groups, which remains limited so far.

\paragraph{Specific families of groups.} As examples of specific groups for which it would be interesting to understand median fixed-point properties, we have already mentioned in the introduction Cornulier's conjecture about lattices in Lie groups. We have also discussed, earlier in the section, (classical) small cancellation groups. Let us mention another open question, probably well-known by specialists (at least in part):

\begin{question}
Let $\Sigma_g$ be a closed surface of genus $g \geq 3$. Does the mapping class group $\mathrm{MCG}(\Sigma_g)$ satisfy $(\mathrm{FW})$? $(\mathrm{FW}_\mathrm{fin})$?
\end{question}

\noindent
It has been proved in \cite{MR1339818} (see also \cite{MR2665003}) that mapping class groups cannot act properly on CAT(0) spaces by semi-simple isometries, and a fortiori on median graphs of finite cubical dimension. In fact, according to \cite{MR4574362}, the finiteness of the cubical dimension can be removed. But, up to my knowledge, no median fixed-point property is known for mapping class groups (except for trees). It is worth noticing that $(\mathrm{FW}_\mathrm{fin})$ would imply that mapping class groups do not virtually surjects onto $\mathbb{Z}$, a well-known question which is still open.

\addcontentsline{toc}{section}{References}

\bibliographystyle{alpha}
{\footnotesize\bibliography{FWCC}}

\Address

%

\end{document}